\documentclass[reqno, eucal,16pt]{amsart}
\usepackage[mathscr]{eucal} 
\usepackage[dvipdfmx]{graphicx}
\usepackage[dvips]{color}
\usepackage{amsmath,amsfonts,amssymb,amsthm,amscd}
\usepackage{epic}
\usepackage{eepic}
\usepackage{longtable}
\usepackage{array}
\usepackage{fourier} 
\usepackage{comment}
\usepackage[all]{xy}

\newcommand{\Osc}{\text{Osc}}

\usepackage[font=small]{caption}
\usepackage{ytableau}
\usepackage[vcentermath, enableskew]{youngtab}
\usepackage{hyperref}
\hypersetup{colorlinks,linkcolor={red},citecolor={olive},urlcolor={red}}

\usepackage{tikz}

\usepackage{enumitem}

\newtheorem{theorem}{Theorem}[section]
\newtheorem{lemma}[theorem]{Lemma}

\newtheorem{proposition}[theorem]{Proposition}
\newtheorem{corollary}[theorem]{Corollary}

\theoremstyle{definition}
\newtheorem{definition}[theorem]{Definition}

\newtheorem{remark}[theorem]{Remark}

\def\Z{\mathbb{Z}} 
\def\R{\mathbb{R}} 
\def\N{\mathbb{N}} 

\def\P{\mathbb{P}} 
\def\E{\mathbb{E}} 
\def\eps{\varepsilon}

\DeclareRobustCommand{\cev}[1]{\reflectbox{\ensuremath{\vec{\reflectbox{\ensuremath{#1}}}}}}

\begin{document}

\title[Diffusive Scaling limit of stochastic Box-Ball systems and P\MakeLowercase{ush}TASEP]{Diffusive Scaling limit of stochastic Box-Ball systems and P\MakeLowercase{ush}TASEP}

\author{David Keating}
\address{David Keating, Department of Mathematics, 
University of Illinois, Urbana-Champaign, IL, 61801, USA}
\email{\texttt{dkeating@illinois.edu}}

\author{Minjun Kim}
\address{Minjun Kim, Pohang University of Science and Technology. Pohang, Kyoungbuk 790-784, South Korea.
}
\email{\texttt{minjoonkim@postech.ac.kr}}

\author{Eva Loeser}
\address{Eva Loeser, Department of Statistics and Operations Research, University of North Carolina at Chapel Hill, 204 E Cameron Ave, Chapel Hill, NC, 27514, USA}
\email{\texttt{ehloeser@unc.edu}}

\author{Hanbaek Lyu}
\address{Hanbaek Lyu, Department of Mathematics, 
University of Wisconsin - Madison, WI, 53706, USA}
\email{\texttt{hlyu@math.wisc.edu}}


\begin{abstract}
	We introduce the Stochastic Box-Ball System (SBBS), a probabilistic cellular automaton that generalizes the classic Takahashi-Satsuma Box-Ball System. In SBBS, particles are transported by a carrier with a fixed capacity that may fail to pick up any given particle with a fixed probability $\eps$. This model interpolates between two known integrable systems: the  Box-Ball System (as $\eps\rightarrow 0$) and the PushTASEP (as $\eps\rightarrow 1$). We show that the long-term behavior of SBBS is governed by isolated particles and the occasional emergence of short solitons, which can form longer solitons but are more likely to fall apart. More precisely, we first show that all particles are isolated except for a $1/\sqrt{n}$-fraction of times in any given $n$ steps, and solitons keep forming for this fraction of times. We then show that under diffusive scaling, both SBBS (for any carrier capacity) and PushTASEP converge weakly to semimartingale reflecting Brownian Motions (SRBMs) on the Weyl chamber with explicit covariance and reflection matrices, which are consistent with the microscale relations between these systems. The reflection matrix for SBBS is determined by how 2-solitons behave and exhibit ``solitonic bias'' visible in the diffusive scale. 
	Our proof relies on a new, extended SRBM invariance principle that we develop in this work. This principle can handle processes with complex boundary behavior that can be written as ``overdetermined'' Skorokhod decompositions, which is crucial for analyzing the complex solitonic interaction in SBBS. We believe this tool may be of independent interest.
\end{abstract}

\maketitle


\section{Introduction}
\label{Introduction}

\subsection{The Box-Ball System}\label{subsection:intro1}
In 1990, Takahashi and Satsuma proposed a $1+1$ dimensional cellular automaton of filter type called the 
\emph{soliton cellular automaton}, also known as the \emph{box-ball system} (BBS) \cite{nagai1999soliton,takahashi1990soliton}. 
It is a well-known integrable discrete-time dynamical system, which we denote as  $\left(\eta_{t}\right)_{t\ge 0}$, whose states are binary sequences 
$\eta_{t}:\N\rightarrow\{0,1\}$ with finitely many $1$'s. 
We may think of the states as configurations of balls in boxes where box $x$ contains a ball at stage $t$ if $\eta_{t}(x)=1$ and is empty 
if $\eta_{t}(x)=0$. 
The update rule $\eta_{t}\mapsto \eta_{t+1}$ is 
using a `carrier' of infinite capacity. For each update, it starts at the origin and sweeps 
rightward to infinity. Each time it encounters an occupied box, it \emph{pushes} the ball to the top of her stack. Each time it encounters 
an empty box and its stack is nonempty, it \emph{pops} any ball from its stack into the box. 
As a concrete example, the system initially having balls in boxes $1,2,4,6,7,8,11,13,16$  evolves through time $t=3$ as
\[
\setlength\arraycolsep{3pt}
\begin{array}{*{2}{r|lllllllllllllllllllllllllllll@{\ }}}
	t=0 & & 1&1&0&1&0&1&1&1&0&0&1&0&1&0&0&1&0&0&0&0&0   &0&0&0&0&0& \ldots \\[\jot]
	1 & & 0 & 0 & 1 & 0 & 1 & 0 & 0 & 0& 1 & 1 & 0 &1  & 0 & 1 & 1 & 0 & 1 & 1 & 0 & 0 & 0 & 0 & 0&0&0&0&\ldots \\[\jot]
	2 & & 0 & 0 & 0 & 1 & 0 & 1 & 0 & 0 & 0 & 0 & 1 & 0 & 1 & 0 & 0 & 1 & 0 & 0 & 1 & 1 & 1 & 1 & 0&0&0&0&\ldots \\[\jot]
	3 & & 0 & 0 & 0 & 0 & 1 & 0 & 1 & 0 & 0 & 0 & 0 &1 & 0 & 1 & 0 & 0 & 1  & 0 & 0 & 0 & 0 & 0 & 1&1&1&1&\ldots \\
\end{array}
\]
\vspace{4pt}

The above carrier description of the BBS dynamics 
is known to be obtained by the limit as $q\rightarrow 0$ of the quantum integrable systems known as vertex models \cite{fukuda2000energy,hatayama2001factorization}. Also, BBS is the ultradiscrete limit of the classical Korteweg-de Vries (KdV) equation \cite{inoue2012integrable}. Classical solitons described by the KdV equation persist in BBS in their `ultradiscrete form'. Namely, in the BBS model, a (non-interacting) \emph{$k$-soliton} for each $k\ge 1$ is defined to be a string of $k$ consecutive $1$'s followed 
by $k$ consecutive $0$'s. During each update, such a soliton travels to the right at speed $k$.
In general, a $k$-soliton consists of $k$ 1's followed by $k$ 0's that are not necessarily consecutive. They can be identified by the Takahashi-Satsuma algorithm, see \cite{takahashi1990soliton}.
The physical interpretation is that 
of a traveling wave with velocity equal to its wavelength. 
If a $k$-soliton precedes a $j$-soliton with $j<k$, then the two will eventually collide, resulting in interference. The subsequent states of 
the system depend on the congruence class of their initial distance modulo their relative speed, $k-j$, but solitons are never created or destroyed in 
the course of these interactions. 

\begin{figure}[h]
	\includegraphics[width=1\textwidth]{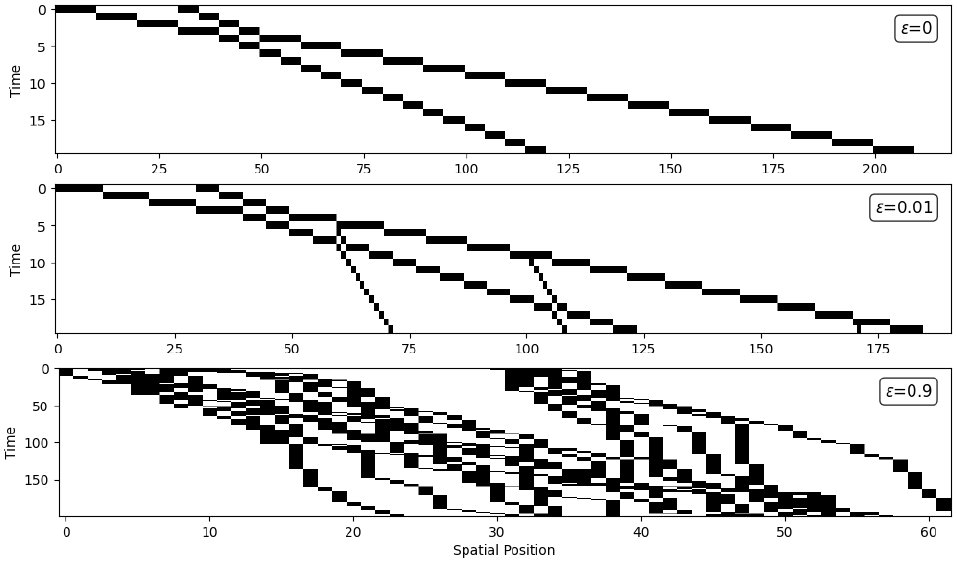}
	\caption{ Sample trajectories of SBBS with $\eps=0, 0.01, 0.9$ with initial configuration $\eta_{0}$ starting with 10 consecutive 1's followed by 20 consecutive 0's followed by 5 consecutive 1's followed by 0's. At $\eps=0$, SBBS coincides with BBS, and the 10-soliton and the 5-soliton swap their momentum with nonlinear scattering in position. At $\eps=0.01$, some balls from these solitons are left behind, creating further interactions. At $\eps=0.9$, initial soliton structures are forgotten rapidly, and the system behaves like a PushTASEP. }
	\label{fig:sim_SBBS_short_term}
\end{figure}

There are various generalizations of BBS. One such generalization we consider in this work is restricting the carrier to have a finite integer capacity $c\ge 1$ 
such that, when its stack contains $c$ balls, any additional balls encountered are skipped. This is known as the capacity-$c$ BBS \cite{hatayama2001factorization,inoue2012integrable}.

\subsection{Stochastic BBS}
We introduce the \textit{stochastic box-ball system} (SBBS), an extension of the Takahashi-Satsuma BBS with a \emph{failure probability} $\eps \in [0,1]$. During a sweep, the carrier independently fails to pick up each ball with probability $\eps$. However, a non-empty carrier still deterministically drops a ball into a newly encountered empty box. A single sweep of this $\eps$-stochastic carrier defines a new random binary configuration. Formally, let $\eta=(\eta_{1},\eta_{2},\dots)\in \{0,1\}^{\mathbb{N}}$ be a binary configuration and $c\in \mathbb{N}\cup \{\infty\}$ the capacity. The \textit{stochastic carrier process} $\Gamma:\mathbb{Z}_{\ge 0}\rightarrow \mathbb{Z}_{\ge 0}$ \textit{with capacity $c$} on $\eta$ is defined by $\Gamma(0)=0$ and:
\begin{align}
	\Gamma(x) =
	\begin{cases}
		\Gamma(x-1) + 1 & \textup{with probability $1-\eps$ if $\eta_{x}=1$ and $\Gamma(x-1)<c$},  \\
		\Gamma(x-1) -1 & \textup{if $\eta_{x}=0$ and $\Gamma(x-1)\ge 1$} \\
		\Gamma(x-1) & \textup{otherwise}.
	\end{cases}
\end{align}
The next BBS configuration $\eta'$ is defined by
\begin{align}
	\eta'(x) = \mathbf{1}_{\{ \Gamma(x)-\Gamma(x-1)=-1 \} }  +  \mathbf{1}_{\{ \Gamma(x)=\Gamma(k-1),\, \eta(x)=1\} },
\end{align}
where, thoughout this paper, we write $\mathbf{1}_{A}$ for the indicator function of the event $A$. The rule above says that  $\eta'(x)=1$ if a ball was dropped at site $x$, or if a ball was at site $x$ ($\eta(x)=1$) and the carrier failed to pick it up. The random evolution map $ T^{c,\eps}: \eta \mapsto \eta'$ defines the SBBS trajectory $(\eta_{t})_{t\ge 0}$, $\eta_{t+1}=T_{\eps}(\eta_{t})$, with initial configuration $\eta_{0}$. We use $d$ to denote the number of balls in $\eta_{0}$.

Instead of keeping track of the full binary configurations, one can equivalently keep track of the positions of the ordered balls.  A \textit{$d$-ball configuration} is an integer vector $x\in \Z^{d}$ with $x_{1}<\cdots<x_{d}$, where each $x_{i}$ is though of as the location of the $i$th ball from left. We identify such $x$ as a map $x:\{1,\dots,d\} \rightarrow \Z$ satisfying $x(1)< \cdots < x(d)$. 
Let $\mathbb{S}^{d}$ denote the set of all $d$-ball configurations. 
From the SBBS trajectory $(\eta_{t})_{t\ge 0}$, we can derive the trajectory $\zeta^{c,\eps}:=(\zeta_{t})_{t \ge 0}$ of ball configuration by setting $\zeta_{t}(i)$ to be the site that holds the $i$th 1 from the left in the binary string $\eta_{t}$ (i.e., $\zeta_{t}(1)=\min\{ x>0\,:\,  \eta_{t}(x)=1 \}$, $\zeta_{t}(2)=\min\{ x> \zeta_{t}(1)\,:\,  \eta_{t}(x)=1 \}$, and so on). The main goal in this paper is to obtain the scaling limit of the process $\zeta^{c,\eps}$.

The basic phenomenology of SBBS is as follows. Analogous to the solitons in BBS, we refer to a run of $k$ consecutive 1's as a $k$-soliton\footnote{In fact, this analogy is imprecise since a $k$-soliton in BBS consists of $k$ 1's followed by $k$ 0's that are not necessarily consecutive. Solitons in BBS can be identified by the Takahashi-Satsuma algorithm, see \cite{takahashi1990soliton}.}. 
Due to the constant probability $\eps$ that the stochastic carrier fails to pick up balls, long solitons cannot persist in the system, and they break up into small pieces. When all balls are separated by distance $\ge 2$, each ball performs an independent random walk with mean speed $1-\eps$. Occasionally, two consecutive balls will meet (being next to each other), forming a 2-soliton that moves with speed 2 during a geometric amount of time (with rate $1-(1-\eps)^{2}$) until the first time the stochastic carrier fails to pick up either of the two balls. If this 2-soliton catches up another ball in front during its lifespan, it can form a longer soliton that has a higher speed. But longer solitons have shorter lifespans, so we should expect to see mostly singleton balls performing independent random walks and occasional short solitons emerging and marching to the right with speed proportional to their length. We can clearly see such long-term dynamics of SBBS in the simulation provided in Figure \ref{fig:sim_SBBS_long_term}.

\begin{figure}[h]
	\includegraphics[width=1\textwidth]{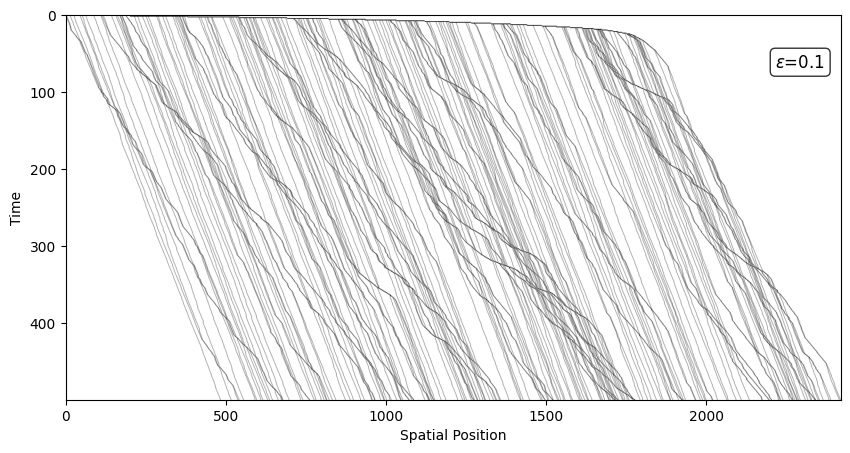}
	\caption{ A sample trajectory of SBBS with infinite capacity, $\eps=0.1$, and initial configuration $\zeta_{0}$ starting with 200 consecutive 1's followed by 0's. The initial soliton falls apart due to the carrier's failure to pick up balls, and the long-term dynamics are governed by singleton balls performing independent random walks and the spontaneous emergence of short solitons.}
	\label{fig:sim_SBBS_long_term}
\end{figure}

\subsection{SBBS and PushTASEP}

Clearly, the stochastic BBS coincides with the original Takahashi-Satsuma BBS \cite{takahashi1990soliton} when $\eps=0$ (see Fig. \ref{fig:sim_SBBS_short_term} top). Interestingly, at the other extreme of $\eps\nearrow 1$, the stochastic BBS becomes the integrable particle process known as PushTASEP \cite{borodin2008large,petrov2020pushtasep}. Briefly, PushTASEP is a continuous-time particle process on the integers in which only one particle may occupy each site and each particle is equipped with an independent exponential clock. When a particle's clock rings, the particle jumps one space to the right. If a particle is currently occupying that site, it is pushed to the right and, in turn, pushes any particle to its immediate right, and so on. 

Now, if $\eps$ is sufficiently close to one, then with high probability in any single sweep, at most one ball is picked up by the carrier, say, at site $x$. If there is a sequence of balls to the right of $x$ up to site $y\ge x$, then the ball will be dropped by the carrier at site $y+1$. As we do not keep track of the identity of the balls but only the ordered tuple of ball positions, this is equivalent to, at the end of the sweep, all balls in the interval $[x,y]$ having been pushed by one to the right just like in PushTASEP (see Fig. \ref{fig:sim_SBBS_short_term} bottom).  This connection is made precise in Theorem \ref{thm: SBBS to PushTASEP}.

Although it is not obvious that SBBS itself is integrable, we see that it interpolates between two types of integrable particle processes:  with $\eps\searrow 0$, we recover the original Takahashi-Satsuma BBS; as $\eps\nearrow 1$ we obtain PushTASEP. Both of these systems are integrable in the sense of having exact formulas for the prelimit systems (see \cite{inoue2012integrable, petrov2020pushtasep}).

\subsection{Overview of main results}

Our first set of results concerns the microscopic behavior of SBBS and PushTASEP. First, we show that the PushTASEP arises as the limit of the SBBS model as the noise parameter $\eps \nearrow 1$ (Theorem \ref{thm: SBBS to PushTASEP}). Next, we show that out of the first $n$ steps, solitons exist only for a $1/\sqrt{n}$-fraction of times (Theorem \ref{thm:upper_bd_local_time_bdry}). As a corollary, we show that the expected location of each ball  
in a $d$-ball SBBS is between $(1-\eps)n$ and $(1-\eps)n+O(\sqrt{n})$, with a sharp asymptotic for the $d=2$ case (Corollary \ref{cor:last_ball}). We obtain similar results for a $d$-ball PushTASEP (Theorem \ref{thm:upper_bd_local_time_bdry_PushTASEP}).

In our second set of results,  we establish that, under diffusion scaling, both models converge to semimartingale reflecting Brownian motions (SRBMs) on a $d$-dimensional convex polyhedron \cite{dai1996existence} known as the  Weyl chamber.   Roughly speaking, an SRBM on an $m$-faced convex polyhedral domain $S\subseteq \R^{d}$ is a stochastic process $\mathcal{W}$ that admits a Skorokhod-type decomposition 
\begin{align}
	\mathcal{W} = \mathcal{X} + R\mathcal{Y}.
\end{align}
where $\mathcal{X}$ is a $d$-dimensional Brownian motion with drift $\theta$, covariance matrix $\Sigma$, and initial distribution $\nu$. The `bulk process' $\mathcal{X}$ gives the behavior of $\mathcal{W}$ in the interior of $S$.  When it is at the boundary of $S$, it is pushed instantaneously toward the interior of $S$ along the direction specified by the columns (each defined up to a positive scalar multiple) of the `reflection matrix' $R\in \R^{d\times m}$ and an associated `pushing process' $\mathcal{Y}$. We say such process $\mathcal{W}$ is an SRBM associated with $(S, \theta, \Sigma, R, \nu)$ (see Def. \ref{def: srbmgeneral}). 

Now define two $d \times (d-1)$ matrices, $R_{PT}$ and $\hat{R}^{c,\eps}$, by 
\begin{align}\label{eq:sigma_R}
	(R_{PT})_{ij} := \mathbf{1}_{\{i=j+1\}}, \qquad 
	(\hat R^{c,\eps})_{ij} :=\begin{cases}
		\eps \mathbf{1}_{\{i=j+1\}} = \eps(R_{PT})_{ij} \quad &\text{if $c=1$,}\\
		(1-\eps)\mathbf{1}_{\{i=j\}}+\mathbf{1}_{\{i=j+1\}} \quad &\text{if $c\ge2$.}
	\end{cases} 
\end{align}
In Theorem \ref{thm: SRBM_weak_convergence}, we show that the SBBS with time scale $\frac{1}{1-\eps}$, error probability $\eps$, and capacity $c$ converges weakly in the diffusive scaling to the SRBM on the $d$-dimensional Weyl chamber
\begin{align}\label{eq:def_Weyl_chamber}
	S: = \{x \in \R^d: x_1 \le x_2 \le \cdots \le x_d \} = \{x \in \R^d: n_i \cdot x \ge 0 \text{ for all $i = 1,2,\dots, d-1$}
	\}, 
\end{align}
where $n_i : = (0,\dots,0,-1,1,0,\dots,0)$ whose $i$th coordinate is $-1$,
and with zero drift, initial distribution the point mass at the origin, covariance matrix $\eps I_d$, where $I_d$ is the $d$-dimensional identity matrix, and reflection matrix $\hat R^{c,\eps}$. In Theorem \ref{thm: pushtasepSRBM}, we show an analogous SRBM limit of PushTASEP. Namely, PushTASEP converges weakly in diffusive scaling to the SRBM with data $(S; \, \mathbf{0}, \, I_d, R_{\textup{PT}}, \delta_{\mathbf{0}})$. These results are summarized in the following diagram:
\begin{align}\label{eq:comm_diag}
	\xymatrix{
		n^{-1/2}(\zeta_{\lfloor \frac{nt}{1-\eps} \rfloor}^{c\ge 1,\eps}-nt\mathbb{1}_d)
		\ar[d]^-{n\to\infty} 
		\ar[r]_-{\eps\nearrow 1} 
		&   n^{-1/2}(\xi_{nt}-nt\mathbb{1}_d) 
		\ar[d]^-{n\rightarrow \infty} 
		&  n^{-1/2}(\zeta_{\lfloor \frac{nt}{\eps(1-\eps)} \rfloor}^{c=1,\eps}-\frac{nt}{\eps}\mathbb{1}_d)
		\ar[dl]^-{n\to\infty} 
		\\
		\textup{SRBM}(S; \, \mathbf{0}, \, \eps I_d, \hat{R}^{c,\eps}, \delta_{\mathbf{0}}) 
		\ar[r]^-{\textup{$\eps\nearrow 1$}} 
		& \textup{SRBM}(S; \, \mathbf{0}, \, I_d,  R_{\textup{PT}}, \delta_{\mathbf{0}}) 
	}
\end{align}
where $\mathbb{1}_{d}$ denotes the $d$-dimensional vector of all ones, $\zeta_t^{c,\eps}$ denotes the SBBS with capacity $c\ge 1$ and error probability $\eps$, and $\xi_t$ is PushTASEP. 
Interestingly, the relation between SBBS and PushTASEP with the procedure $\eps \nearrow 1$ passes through the diffusive limit, which is shown by the left part of the commutative diagram above. In addition, the unit-capacity SBBS for a fixed $\eps\in (0, 1)$ has the same diffusive limit as the PushTASEP (Corollary \ref{cor: SRBM_weak_convergence}).

The SRBM limit for PushTASEP and SBBS with $c=1$ can be constructed recursively as a system of $d$ reflecting Brownian motions, where the first coordinate is a standard Brownian motion, the second coordinate is a standard Brownian motion that reflects off the trajectory of the first coordinate, and so on (see Remark \ref{rmk:SRBM_recursive}). For SBBS with $c\ge 2$, the columns of the corresponding reflection matrix $\hat{R}^{c,\eps}$ show that 
the reflection direction is slanted in the tangential direction, converging to a tangent vector as $\eps\nearrow 0$. Hence, the corresponding two balls move nearly as a single 2-soliton (they seem to have speed 1 instead of 2 due to the centering), demonstrating ``solitonic bias'' of SBBS persisting in the diffusive scale. See Figure  \ref{fig:SBBS_d2}.

\begin{figure}[h]
	\includegraphics[width=1\textwidth]{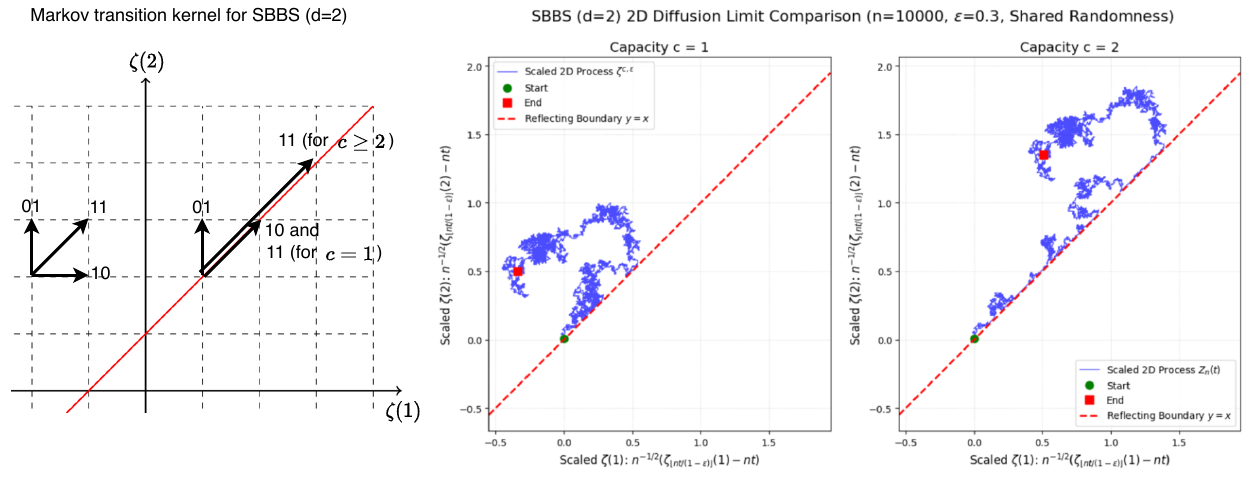}
	\vspace{-0.7cm}
	\caption{
		Left: Transition kernel for SBBS $\zeta^{c,\eps}$ with $d=2$. Each nonzero transition is shown with the corresponding independent coin flips with failure probability $\eps$ for the two balls as a binary string. Right: Sample paths of $\zeta^{c,\eps}$ in diffusive scaling for capacities $c=1$ and $c=2$ with shared coin flips for all times. The mean reflection vector $\hat{R}^{c,\eps}$  is $(0, \eps)^{\top}$ for $c=1$ and $(1-\eps, 1)^{\top}$ for $c= 2$, respectively. The vector for $c = 2$ is slanted more in the tangential direction, 
		Demonstrating the stronger solitonic bias.
	} 
	\label{fig:SBBS_d2}
\end{figure}

Lastly, for a unified treatment of all the aforementioned SRBM limits, we extend the classical invariance principle for SRBM due to \cite{williams1998invariance, kang2007invariance}  (Theorem \ref{thm: main}). 
The novelty of our extension is that we allow an arbitrary number $k$ of boundary cells—possibly significantly exceeding the number of faces, denoted by $m$, of the domain $S$—on which the process of interest may have distinct boundary behaviors. Namely, suppose we have a family of processes $(\zeta^{n})_{n\ge 1}$ on $S$ admitting the following ``overdetermined'' Skorokhod decomposition 
\begin{align}
	\zeta^{n} = X^{n} + RY^{n},
\end{align}
where $X^{n}$ is a suitable bulk process on $\R^{d}$, 
$R$ is a $d\times k$ rectangular reflection matrix, and $Y^{n}$ is the pushing process in dimension $k$. The correction $\zeta^{n}-X^{n}$ may occur in all $k\ge m$ boundary cells, not only on the $m$ faces of the domain.  In Theorem \ref{thm: main}, we show that, under some additional assumptions, only the $m$ columns of $R$ corresponding 
with potential applications to analyzing diffusive limits of other discrete interacting particle systems.

\subsection{Related work}

In the most general form of the BBS, each site accommodates a semistandard tableau of rectangular shape with letters from  $\{0,1,\cdots,\kappa \}$ and the time evolution is defined by successive application of the combinatorial $R$ (cf. \cite{fukuda2000energy, hatayama2001, kuniba2006crystal,inoue2012integrable}). When the carrier is set as the $1\times c$ semistandard tableau and $\kappa=1$, we obtain the BBS with capacity-$c$ carrier.
The BBS is known to arise both from the quantum and classical integrable systems by the procedures called crystallization and ultradiscretization, respectively. This double origin of the integrability of BBS lies behind its deep connections to quantum groups, crystal base theory, solvable lattice models, the Bethe ansatz, soliton equations, ultradiscretization of the Korteweg-de Vries equation, tropical geometry, and so forth; see, for example, the review \cite{inoue2012integrable} and the references therein. BBS enjoys integrable structures in many ways, for example via the classical Kerov–Kirillov–Reshetikhin bijection  \cite{kerov1988combinatorics, kirillov1988bethe} 
and the more recent slot decomposition \cite{ferrari2018soliton} and the seat number configuration 
\cite{mucciconi2024relationships, suda2023seat}.

BBS with random initial configuration is an active topic of research in the integrable probability community. The scaling limit and phase transition of the invariant Young diagram for the original BBS was obtained in \cite{levine2020phase}. The row-scaling of the Young diagram for the generalized BBS was obtained in \cite{kuniba2018randomized,kuniba2020large}, while its column scaling for the multi-color BBS was established in \cite{lewis2024scaling}. In particular, the infinite-capacity carrier process at the critical phase is known to converge in diffusive scaling to the reflecting Brownian motion for $\kappa=1$ \cite{levine2020phase} and more generally to an explicit SRBM for $\kappa\ge 1$ \cite{lewis2024scaling}. Dynamic properties of the solitons in the extended BBS on the whole line $\Z$ have also been extensively studied recently, including characterization of the invariant measures  \cite{ferrari2018bbs,ferrari2020bbs, croydon2019invariant,croydon2023dynamics}, scaling limits of a tagged soliton \cite{croydon2021generalized, olla2024scaling}, and generalized hydrodynamic limit 
\cite{kuniba2020generalized, kuniba2022current, kuniba2021generalized, croydon2021generalized}. 
However, all these works consider the deterministic BBS dynamics with random initial configuration, and we believe this work is the first to introduce stochasticity into the BBS dynamics itself.

Various asymptotic and scaling limit analyses of PushTASEP have been obtained previously \cite{borodin2014anisotropic,ferrari2014perturbed,gorin2015limits}. For the $d$-particle PushTASEP, the diffusive scaling limit is known to be $d$ reflecting Brownian motions constructed recursively as follows. The first (leftmost) particle performs an independent simple random walk in PushTASEP and thus converges to an independent Brownian motion under diffusive scaling. The second particle converges to a Brownian motion that reflects off the trajectory of the first, and the construction proceeds similarly for subsequent particles.  
Gorin and Shkolnikov \cite{gorin2015limits} show that the multilevel extension of TASEP as a process on non-interlacing particle configurations (also known as the Gelfand-Tsetlin patterns) converges weakly to the system of reflecting Brownian motion introduced by Warren \cite{warren2007dyson}. This extended process projects to PushTASEP by considering the right-most particles, and Warren's process under the same projection gives the system of reflecting Brownian motions described above. 
In this work, we use SRBMs on the Weyl chamber as a class of limit objects for SBBS and PushTASEP, which specializes to the reflecting Brownian motions mentioned above for SBBS with unit capacity $c=1$ and PushTASEP. 

\subsection{Organization} 

We state the main results in this paper in Section \ref{sec: main results}. Specifically, results on microscopic behaviors are stated in Subsection \ref{sec:results_micro}, diffusive scaling behaviors in Subsection \ref{sec: diffusive scaling gap process}, and invariance principle for SBRM via overdetermined Skorokhod decomposition in Subsection \ref{sec:results_SRBM}. In Section \ref{sec:open_problems}, we state some conjectures and discuss open problems. 
In Section \ref{sec:proof_micro}, we show the microscale relation between SBBS and PushTASEP stated in Theorem \ref{thm: pushtasepSRBM}. Next in Section \ref{sec:2-ball}, we prove Theorem \ref{thm:upper_bd_local_time_bdry},  Corollary \ref{cor:last_ball}, and Theorem \ref{thm:upper_bd_local_time_bdry_PushTASEP}. We do so by analyzing the associated gap process that keeps track of the inter-distances between the balls. 
In Section \ref{sec:proof_micro}, we prove the results on the diffusive scaling limit of SBBS and PushTASEP stated in Subsection \ref{sec: diffusive scaling gap process} by invoking the extended SRBM invariance principle in Theorem \ref{thm: main}. This latter result will be proved in the last section, Section \ref{sec:proof_SRBM_invariance}. 

\section{Statement of main results}\label{sec: main results}

In this paper, we develop a unified framework connecting two models: the SBBS and the PushTASEP, along with their diffusive scaling limits.


\subsection{Microscopic behavior}
\label{sec:results_micro}

In this subsection, we state several results on the microscopic behavior of SBBS and PushTASEP. First, we show that SBBS with error probability $\eps\nearrow 1$ and arbitrary capacity $c$ converges to the PushTASEP.

\begin{theorem}[PushTASEP is $\eps\nearrow 1$ scaling limit of SBBS]\label{thm: SBBS to PushTASEP}
	The continuous-time version of the $d$-ball SBBS trajectory with capacity $c = 1,2,\dots, \infty$ and error probability $\eps \in (0,1)$,  $(\zeta^{c,\eps}_{\lfloor \frac{t}{1-\eps} \rfloor})_{t \ge 0}$, 
	converges weakly to the $d$-particle PushTASEP as $\eps \nearrow 1$.   
\end{theorem}

The proof of Theorem \ref{thm: SBBS to PushTASEP} above is based on the following simple observation. If $\eps$ is very close to one, then during each update of SBBS, at most one ball is picked up by the carrier. Let $x$ denote the location of this ball. If the site $x+1$ is vacant, then the ball is dropped at $x+1$ just as in PushTASEP. If $x+1,\dots,y$ are occupied for some $y\ge x+1$, then the carrier is not likely to pick up any balls in the interval $[x+1,y]$, so the ball picked up at site $x$ is dropped at $y+1$. Since we do not distinguish the identities of the balls, we may as well view it as the ball at $x$ pushing all the balls in $[x+1,y]$, again exactly as in PushTASEP. This happens on average one in every $1/(1-\eps)$ iterations, so after rescaling time by $t/(1-\eps)$ we get PushTASEP as the limiting process.

\begin{remark}[Unit capacity SBBS and PushTASEP]\label{rmk:unit_SBBS}
	
	The unit capacity $(c=1)$ SBBS, for any fixed $\eps\in(0,1)$, can be viewed as a version of discrete-time PushTASEP with a "solitonic bias." Indeed, suppose $x\le y$, and that we have a block of 1s on the interval $[x,y]$, and sites $x-1$ and $y+1$ are vacant. The unit capacity carrier is empty right before it starts interacting with site $x$. The unit capacity carrier is empty right before it starts interacting with site $x$. If it picks up the ball at $x$, it will carry that ball to the first empty site to its right, which is $y+1$. Since we do not keep track of the identities of the balls, this can be viewed as pushing the entire block just as in PushTASEP.
	
	However, while every ball is equally likely to make a jump in PushTASEP, the probability that the $i$th ball in the interval $[x,y]$ is picked up by the carrier decays geometrically as $\eps^{ 
		i-1}(1-\eps)$. In other words, while a "soliton" (i.e., a run of 1s) in PushTASEP is likely to break at any point, in the unit capacity SBBS, there is a bias towards preserving the soliton. If it does break, there is a geometric bias for the break point to be toward the left.

\end{remark}

Due to the constant probability of the carrier missing a ball, one can expect that solitons will not last long in SBBS with fixed error probability $\eps\in (0,1)$. Indeed, in Theorem \ref{thm:upper_bd_local_time_bdry}, we show that out of the first $n$ steps, solitons exist only for a $1/\sqrt{n}$-fraction of times. We can obtain a precise asymptotic for $d=2$. This shows the diffusive nature of SBBS. 

\begin{theorem}[Solitons exist only for $1/\sqrt{n}$-fraction of times]\label{thm:upper_bd_local_time_bdry}
	Let $\zeta$ be the $d$-ball SBBS with capacity $c\ge 1$ and error probability $\eps\in (0,1)$. Let $N^{(d)}_n$ denote the number of times that some two balls are within distance one up to time $n$. Then $\mathbb{E}[N_n^{(d)}] = \Theta(\sqrt{n})$. 
	\begin{align}\label{eq:N_n_2_ball}
		\E[N_{n}^{(2)}]=(1+o(1))\sqrt{\frac{2\eps^{3}(1-\eps)^{3} n}{\pi }}.
	\end{align}
\end{theorem}

Next, we consider how far the $i$th ball from the left in a $d$-ball SBBS travels in a given $n$ steps. According to Theorem \ref{thm:upper_bd_local_time_bdry}, except for a $1/\sqrt{n}$ fraction of times, the $i$th ball will be in isolation, so it moves with speed $1-\eps$, the probability of being picked up by the carrier. When not in isolation, it can move at a faster speed depending on the length of the soliton that contains the $i$th ball and how many balls from this soliton are picked up by the carrier. Since the number $d$ of balls is fixed, the speed of any particle is uniformly upper-bounded by $d$. Thus the $i$th ball should be at $(1-\eps)n+O(\sqrt{n})$ at time $n$. A precise statement is given in Corollary \ref{cor:last_ball} below.

\begin{corollary}[Expected location of the  balls]\label{cor:last_ball}
	Keep the same setting in Theorem \ref{thm:upper_bd_local_time_bdry}. 
	For each $1\le i \le d$, 
	\begin{align}\label{eq:last_ball_gen}
		\E[\zeta_{n}(i)-\zeta_{0}(i)]  
		&= (1-\eps)n+O(\sqrt{n}).
	\end{align}
	In particular, for $d=2$, 
	\begin{align}\label{eq:zeta_2_ball}
		\E[\zeta_{n}(1)-\zeta_{0}(1)] 
		&= (1-\eps)n + (1-\eps)^{2}\mathbf{1}(c\ge 2) \sqrt{\frac{2\eps^{3}(1-\eps)^{5} n}{\pi }} 
		+ o(\sqrt{n}), \\
		\E[\zeta_{n}(2)-\zeta_{0}(2)] 
		&= (1-\eps)n + (\mathbf{1}(c\ge 2)+\eps\mathbf{1}(c=1)) \sqrt{\frac{2\eps^{3}(1-\eps)^{5} n}{\pi }} 
		+ o(\sqrt{n}).
	\end{align}
\end{corollary}

For $d=2$, notice that there is an additional order $\sqrt{n}$ term in the expected distance traveled by the first ball when $c\ge 2$, which comes from the solitonic interaction between the two balls when in contact for a $1/\sqrt{n}$-fraction of times. 
This indicates that the trajectory of the first ball cannot converge to a Brownian motion with zero drift, unlike the first ball in PushTASEP. 


We obtain similar results for the PushTASEP, stated in Theorem \ref{thm:upper_bd_local_time_bdry_PushTASEP} below.

\begin{theorem}\label{thm:upper_bd_local_time_bdry_PushTASEP}
	Let $\partial \mathbb{S}^{d}:=\{ x\in \Z^{d}\,:\, \text{$x_{1} < \cdots < x_{d}$ and $x_{i+1}-x_{i}=1$ for some $i=1,\dots,d-1$}\}$. 
	Consider a configuration with $d$ balls and run the PushTASEP dynamics $\xi$ for time $T$. Let $N_T^{(d)}$ be the length of time up to time $T$ that $\xi$ spends in the set $\partial \mathbb{S}^{d}$. Then $\mathbb{E}[N_T^{(d)}] = \Theta(\sqrt{T})$. Furthermore, the location $\xi_t(i)$ of the $i$th particle at time $t$ in PushTASEP satisfies 
	\begin{align}\label{eq:PushTASEP_rightmost}
		\E[\xi^{(i)}_{T}]  
		&= T+O(\sqrt{T}).
	\end{align}
\end{theorem}

\subsection{Diffusive Scaling limit of SBBS and PushTASEP} \label{sec: diffusive scaling gap process}

The $d$-ball SBBS $\zeta^{c,\eps}$ is a process on the space $\mathbb{S}^{d}$ of $d$-ball configurations consisting of the integer points in the interior of the Weyl chamber $S$. 
When all balls are separated by a distance of at least two, its transition kernel is given by a symmetric random walk on $\Z^{d}$. Namely, let $\eta_{t}^{(i)}$, for $t=0,1,\dots$ and $i=1,2,\dots,d$, be i.i.d. $\textup{Bernoulli}(1-\eps)$ variables and write $\eta_{t}:=(\eta_{t}^{(1)},\dots,\eta_{t}^{(d)})$. Each $\eta_{t}^{(i)}$ is the randomness the carrier uses to pick up the $i$th ball at time $t$ when not fully loaded. 
Let $\mathbf{e}_{1},\dots,\mathbf{e}_{d}$ be the standard basis vectors in $\R^{d}$. Then we define the `bulk process' $X$ on $\Z^{d}$ by 
\begin{align}\label{eq:def_bulk_process}
	\Delta X = X_{t+1}-X_{t}  = \sum_{i=1}^{d}  \eta_{t}^{(i)}  \mathbf{e}_{i}.
\end{align}
When all coordinates of $W$ are distinct (i.e., all balls are separated), then the one-step evolution of $W$ agrees with the one-step evolution of $X$. Clearly $X$ defines a random walk on $\Z^{d}$ (with non-nearest-neighbor jumps) with $\E[\Delta X]=(1-\eps) \mathbb{1}_{d}$, and it is easy to check that the covariance matrix $\Sigma$ for the one-step transition for $X$ is given by
\begin{align}\label{eq:X_cov}
	\Sigma =   \E[  (\Delta X-(1-\eps)\mathbb{1}_d)^{\top} (\Delta X-(1-\eps)\mathbb{1}_d) ] = \eps(1-\eps) I_{d}.
\end{align}
In particular, in the diffusive scaling, $X-(1-\eps)\mathbb{1}_{d}$ converges to the Brownian motion on $\R^{d}$ with zero drift and covariance matrix $\eps(1-\eps) I_{d}$. 
However,  the actual process $\zeta^{c,\eps}$ can significantly differ from the bulk process $X$ due to solitonic interactions. Therefore, we need to analyze the correction process $\zeta^{c,\eps}-X$. A major role of this process is to confine the domain of $\zeta^{c,\eps}$ to $\mathbb{S}^{d}$, whereas  $X$ lives on the whole integer lattice $\Z^{d}$.

The following result establishes the full diffusive scaling limit of  SBBS with arbitrary capacity $c\in \mathbb{N}\cup \{\infty\}$. 

\begin{theorem}[Diffusive scaling limit of SRBM]\label{thm: SRBM_weak_convergence}
	Let $(\zeta_{\lfloor t\rfloor})_{t\ge 0}$ denote the continuous-time $d$-ball SBBS with capacity $c\ge 1$ and error probability $\eps\in (0,1)$. Then as $n\rightarrow\infty$, 
	\begin{align}\label{eq: time scale}
		(n^{-1/2}(\zeta_{\lfloor \frac{nt}{1-\eps} \rfloor}-nt \mathbb{1}_{d})\,: \, t\ge 0 ) \Rightarrow \mathcal{W}\in C([0,1]),
	\end{align}
	where $\mathcal{W}$ is an SRBM associated with data $(S, \mathbf{0}, \eps I_d, \hat{R}^{c,\eps}, \delta_{\mathbf{0}})$ with $S = \{x \in \R^d: x_1 \le x_2 \le \cdots \le x_d \}$ the $d$-dimensional Weyl chamber and     $\hat{R}^{c,\eps}$ as 
	in \eqref{eq:sigma_R}.
\end{theorem}

The $d\times (d-1)$ reflection matrix $\hat{R}^{c,\eps}$ in \eqref{eq:sigma_R} arises from the expected reflections at codimension-$1$ boundary faces of the Weyl chamber corresponding to $x_{i}=x_{i+1}$ for $i=1,\dots,d-1$. 
More precisely, the $i$th column $\hat{R}^{c,\eps}_{i}$ represents the expected reflection when the $i$th and $(i+1)$th balls are adjacent and interact with each other, while every other ball is sufficiently separated so that it does not interact with any others. As we will show, other higher-order solitonic interactions vanish under the diffusive scaling. 
The $i$th row of $\hat{R}^{c,\eps}$ encodes the expected instantaneous bias on the $i$th ball from all boundary cells. For instance, when $c\ge 2$, the first row is $(1-\eps,0,\dots,0)$, indicating the expected ``solitonic boost'' of $1-\eps$ on the first face $x_{1}=x_{2}$. Notice that, however, it is zero when $c=1$, so the first ball does not get any boost within unit capacity SBBS. In fact, it is easy to observe that the first ball in that case performs a symmetric simple random walk on $\Z$.

Next, we show an analogous diffusive scaling limit result for PushTASEP.

\begin{theorem}[Diffusive scaling limit of PushTASEP]\label{thm: pushtasepSRBM}
	Let $(\xi_t)_{t \in \mathbb{R}_{\geq 0}}$  PushTASEP on $d$ particles. Then as $n\rightarrow \infty,$
	$$(n^{-1/2} (\xi_{nt} - nt \mathbb{1}_{d}) : 0 \leq t \leq 1) \rightarrow \tilde{\mathcal{W}} \in C([0,1]),$$ where $\tilde{\mathcal{W}}$ is an SRBM associated with $(S, \boldsymbol{0}, I_{d}, R_{\textup{PT}}, \delta_{\boldsymbol{0}})$ 
	with $S$ the $d$-dimensional Weyl chamber \eqref{eq:def_Weyl_chamber} and  $R_{\textup{PT}}$ is given 
	in \eqref{eq:sigma_R}. 
\end{theorem}

\noindent By comparing Theorems \ref{thm: SRBM_weak_convergence} and \ref{thm: pushtasepSRBM}, we see that either taking $\eps\nearrow 1$ or $c=1$, the SRBM limit of SBBS coincides with the SRBM limit of the PushTASEP. Recall that the procedure $\eps\nearrow 1$ also sends SBBS to PushTASEP as in Theorem \ref{thm: pushtasepSRBM} without diffusive scaling. Hence, this justifies the commutative diagram in \eqref{eq:comm_diag} for $\eps\nearrow 1$. 

It is also interesting to consider the unit capacity case $c=1$. In Remark \ref{rmk:unit_SBBS}, we noted that SBBS with $c=1$ can be thought of as a discrete-time PushTASEP with additional solitonic bias. However, such bias vanishes in diffusive scaling and the unit capacity SBBS with \textit{fixed} error probability converges weakly to PushTASEP. The following corollary is an immediate consequence of Theorems \ref{thm: SRBM_weak_convergence} and \ref{thm: pushtasepSRBM}. 

\begin{corollary}[Unit capacity SRBM and PushTASEP in diffusive scaling]\label{cor: SRBM_weak_convergence}
	Keep the same setting as in Theorem \ref{thm: SRBM_weak_convergence} and further assume $c=1$. Let $\tilde{\mathcal{W}}$ denote the SRBM limit of the gap process of PushTASEP in Theorem \ref{thm: pushtasepSRBM}. Then as $n\rightarrow\infty$, 
	\begin{align}\label{eq:unit_SBBS_diffusion}
		(n^{-1/2}(\zeta_{\lfloor \frac{nt}{\eps(1-\eps)} \rfloor} - (nt/\eps)\mathbb{1}_{d} ) \,: \, t\ge 0 ) \Rightarrow \tilde{\mathcal{W}}\in C([0,1]).
	\end{align}
\end{corollary}

\begin{remark}[Coordinate-wise recursive construction of SRBM on the Weyl chamber]
	\label{rmk:SRBM_recursive}
	The SRBM $\tilde{\mathcal{W}}=(\tilde{\mathcal{W}}(1),\cdots,\tilde{\mathcal{W}}(d))$ in Theorem \ref{thm: pushtasepSRBM} and Corollary \ref{cor: SRBM_weak_convergence} can be described recursively as follows: $\tilde{\mathcal{W}}(1)$ is a standard Brownian motion, $\tilde{\mathcal{W}}(d)$ is a standard Brownian motion that reflects off the trajectory of $W^{1}$, and so on. To see this, we note that it satisfies the Skorokhod decomposition $\tilde{\mathcal{W}}=\mathcal{X}+ R_{\textup{PT}} \mathcal{Y}$, where $\mathcal{X}$ is the standard $d$-dimensional Brownian motion, where its coordinates are independent standard Brownian motions in $\R$. Furthermore, the reflection matrix $R_{\textup{PT}}$ is strictly lower triangular, so we can restrict the decomposition to the first $r$ coordinates for each $r=1,\dots,d-1$. Specifically, the $r$-dimensional process $(\tilde{\mathcal{W}}(1),\cdots,\tilde{\mathcal{W}}(r))$ is an SRBM on the $r$-dimensional Weyl chamber $\{x\in \R^{r}\,:\, x_{1}\le \cdots \le x_{r}\}$ driven by the $r$-dimensional standard Brownian motion and reflection matrix the first $r\times r$ submatrix of 
	$R_{\textup{PT}}$. Applying this observation for $r=1$ shows that  $\mathcal{W}(1)=\mathcal{X}(1)$ is the standard Brownian motion. For $r=2$, $(\tilde{\mathcal{W}}(1),\tilde{\mathcal{W}}(2))$ is the SRBM on the 2-dimensional Weyl chamber where the reflection vector on the boundary $x_{1}=x_{2}$ is $(0,1)^{\top}$, which precisely means that $\tilde{\mathcal{W}}(2)$ is a standard Brownian motion that reflects off the trajectory of $\tilde{\mathcal{W}}(1)$. Proceeding inductively results in the claimed recursive construction of $\tilde{\mathcal{W}}$. 
	
	In general, one can obtain a similar coordinate-wise recursive construction of an SRBM on the $d$-dimensional Weyl chamber if and only if the covariance matrix $\Sigma$ is diagonal and the reflection matrix $R$ is strictly lower-triangular. For instance, the SRBM limit of SBBS in Theorem \ref{thm: SRBM_weak_convergence} for $c\ge 2$ has a reflection matrix $\hat{R}^{c,\eps}$ that has positive diagonal entries, so it cannot be constructed recursively. For instance, the first coordinate is not an independent Brownian motion since the reflection off the first boundary $x_{1}=x_{2}$ is not orthogonal to the axis (see Fig. \ref{fig:SBBS_d2}). 
\end{remark}

\subsection{Overdetermined Skorokhod decomposition and extended SRBM invariance principle}
\label{sec:results_SRBM}

In order to prove the SRBM limits of SBBS and PushTASEP stated in Theorems \ref{thm: SRBM_weak_convergence} and \ref{thm: pushtasepSRBM}, we develop a unified invariance principle for SRBM that allows the number of distinct boundary behaviors to significantly exceed the dimension of the process in the pre-limit. 

We first recall the definition of SRBM on a convex polyhedron $S$ given by a matrix $N=[n_{1},\cdots,n_{m}]$ of unit normal vectors and a vector $b=(b_{1},\dots,b_{m})$ as 
\begin{align}\label{eq:def_polyhedron}
	S = \{x \in \R^d: 
	N x \ge b
	\}
\end{align}
where the inequality above holds entrywise (see \cite{dai1996existence}). In this case $S$ has $m$ faces $F_{1},\dots,F_{m}$ defined by $F_{i}:= \{x\in S\,:\, n_{i}\cdot x = b_{i} \}$ for $1\le i \le m$.

\begin{definition}[semimartingale reflecting Brownian Motion]\label{def: srbmgeneral}
	Let $S$ be the convex polyhedron in \eqref{eq:def_polyhedron}. Given a Borel probability measure $\nu$ on 
	$S$, an SRBM associated with the data $(S, \theta, \Gamma,R,\nu)$ is an $\{\mathcal{F}_t\}$-adapted, $d$-dimensional process $W$ defined on some filtered probability space $(\Omega, \mathcal{F}, \mathcal{F}_t, P)$ such that 
	\begin{enumerate}
		\item $W= X+RY$, $P$-a.s.,
		\item $P$-a.s., $W$ has continuous paths and $W(t) \in S$ for all $t \geq0,$
		\item under $P,$
		\begin{itemize}
			\item $X$ is a $d$-dimensional Brownian motion with drift vector $\theta,$ covariance matrix $\Gamma,$ and $X(0)$ has distribution $\nu,$
			\item $\{X(t)-X(0)-\theta t, \mathcal{F}_t, t\geq 0\}$ is a martingale,
		\end{itemize}
		\item $Y$ is an $\mathcal{F}_t$-adapted, $d$-dimensional process such that $P$-a.s., for $i=1,...,m,$
		\begin{itemize}
			\item $Y_i(0)=0,$
			\item $Y_i$ is continuous and non-decreasing,
			\item  $\int_0^{t} \mathbf{1}_{ F_i } (W(s)) dY_i(s)=Y_i(t)$ for all $t \ge 0$. 
		\end{itemize}
	\end{enumerate}
\end{definition}

In the classical SRBM setting on the nonnegative orthant $\mathbb{R}^d_{\ge 0}$, the reflection matrix is a $d \times d$ square matrix, where the $i$th column corresponds to the reflection vector on the boundary face $\{x \in \mathbb{R}^d_{\ge 0}: x_i = 0\}$. 
If $R=I-Q$  for some nonnegative matrix $Q$ with spectral radius less than one, then such $W$ is unique (pathwise) for possibly degenerate $\Sigma$ when $S=\R^{d-1}_{\ge 0}$ 
\cite{harrison1981reflected}. If $\Sigma$ is non-degenerate and $S$ is a polyhedron, a necessary and sufficient 
A condition for the existence and uniqueness of such SRBM is that $R$ satisfies the `$\mathcal{S}$-condition' \cite{dai1996existence}, which we recall below:
\begin{definition}[maximal subset]
	Let $S$ be the convex polyhedron in \eqref{eq:def_polyhedron} with faces $F_{1},\dots,F_{m}$. 
	We say that a non-empty subset $I \subseteq \{1,2,\dots, \}$ is maximal if $F_I:=\cap_{i \in I}F_i \neq \emptyset$ and $F_I \neq F_{I'}$ for any $I' \supsetneq I$.
\end{definition}
\begin{definition}[$\mathcal{S}$-condition] 
	Let $S$ be the convex polyhedron in \eqref{eq:def_polyhedron}. We say a $d \times m$ matrix $R$ satisfies the \textit{$\mathcal S$-condition} with respect to  $S$ if and only if one of the following conditions, which are equivalent, holds: 
	\begin{itemize}
		\item[(a)] For each maximum subset 
		$I \subseteq \{1,2,\dots, m\}$, the associated principal matrix $(NR)_I$ of $NR$ has a vector $\lambda_I>\mathbf 0$ such that $(NR)_I\lambda_I>\mathbf 0$.
		\item[(b)] For each maximum subset $I \subseteq \{1,2,\dots, m\}$, the associated principal matrix $(NR)_I$ of $NR$ has a vector $\lambda_I>\mathbf 0$ such that $\lambda^\top_I(NR)_I>\mathbf 0$.
	\end{itemize}
\end{definition}

For our purpose, we need an invariance principle that can be used to establish the diffusive scaling limit of discrete processes $W^{n}$ living on the integer points $S^{d}:=S\cap \Z^{d}$ in $S$ toward the SRBM living on $S$. While $S$ has $m$ boundary faces, we allow $W^{n}$ to exhibit $k\ge m$ different behaviors near the boundary of $S$. That is, we assume $S^{d}$ decomposes into the disjoint union of the `interior' $S_{0}^{d}$ and $k$  `boundary cells': 
\begin{align}
	S^{d}= S^{d}_{0} \cup (\partial S^{d}_{1}\cup \cdots \cup \partial S^{d}_{k}).
\end{align}
Accordingly, we further assume $W^{n}$ admits  the following ``overdetermined'' Skorokod decomposition
\begin{align}\label{eq:def_overdetermined_skorokod}
	W^{n} = X^{n} + R Y^{n} \qquad \textup{where $R\in \R^{d \times k}$ and $Y^{n}\in \R^{k}$}. 
\end{align}
The $k$ columns of the rectangular reflection matrix $R$ give the instantaneous direction of reflection off of the boundary cells $\partial S^{d}_{i}$, where $k$ is allowed to be much larger than the number  $m$ of the faces of $S$. Accordingly, the pushing process $Y^{n}$ reserves one coordinate for each of the $k\ge m$ boundary cells. We say the above decomposition is ``overdetermined'' because the number $k$ of boundary cells can exceed the number $m$ of the faces of $S$. 

A key condition we introduce for our generalized SRBM invariance principle is the following weak notion of 
the $\mathcal{S}$-condition above.

\begin{definition}[Weak $\mathcal S$--condition]\label{def:weakly_completely_S}
	Let $R$ be a $d \times k$ matrix, $S$ be a convex polyhedron in \eqref{eq:def_polyhedron} 
	and $f$ be a function from $\{1,2,\dots, k\}$ to the power set  $\mathcal P(\{1,2,\dots, m\})$. Assume that, for each maximal subset $I \subseteq \{1,2,\dots, m\}$, there exists $\lambda_{I} \in \R^I_{> 0}$ such that $(\lambda_{I}^\top N^IR^I)_j \ge 1$ for all $j \in J_I$, where $N^I:=[n_i]_{i \in I}^\top$ is a $|I| \times d$ matrix consisting of row vectors $n_i$ with $i \in I$, $J_I := \{1 \le j \le k: f(j) \subseteq I\}$, and $R^I:=[R_j]_{f(j) \subseteq I}$ consists of column vectors $R_j$ with $f(j) \subseteq I$.
	Then $R$ is said to be weak $\mathcal S$ with respect to $S$ and $f$. 
\end{definition}

The map $f$ in Def. \ref{def:weakly_completely_S} takes  takes a boundary index $j \in \{1,2,\dots,k\}$ to the set of coordinates $f(j) \subseteq \{1,2,\dots,d\}$ that are ``degenerate'' on the $j$th boundary cell. 
For instance, when $S$ is the nonnegative orthant $\R^{d}_{\ge 0}$, then we may take $f(i)=i$ for $i=1,\dots,d$. Our primary example is when $S$ is the Weyl chamber \eqref{eq:def_Weyl_chamber}, which has $d-1$ faces with normal vectors being the columns of the following $d\times (d-1)$ matrix:
\begin{align}
	N_{ij} = \mathbf{1}_{\{i=j\}} - \mathbf{1}_{\{i=j+1\} }.
\end{align}
In this case, $N$ maps the Weyl chamber to the $d-1$-dimensional orthant $\R^{d-1}_{\ge 0}$ and the boundary cells are mapped to rectangular boxes. 
See the $d=3$ example in Sec. \ref{ex:3ballIncrements} for an illustration. 

One can see that our notion of weak $\mathcal{S}$ condition indeed specializes to $\mathcal{S}$-condition. Take $R$ to be $d \times m$ matrix, which satisfies $\mathcal S$-condition, and $f$ to be the identity function on $\{1,...,m\}$. By the definition of $\mathcal S$-condition, for each maximal subset $I \subseteq \{1,2,\dots, m\}$, there exists $\lambda_I > \mathbf 0$ such that $\lambda_I^\top (NR)_I>\mathbf{0}$. Since $R^I:=[R_j]_{f(j) \in I}=[R_j]_{j \in I}$, $\lambda_I^\top N^IR^I=\lambda_I^\top (NR)_I>\mathbf{0}$, which implies weak $\mathcal S$-condition.

In Theorem \ref{thm: main} below, we provide a sufficient condition for processes admitting the overdetermined Skorokhod decomposition \eqref{eq:def_overdetermined_skorokod} to converge weakly to an explicit SRBM in the diffusive limit. Our result generalizes the classical SRBM invariance principle by Williams \cite{williams1998invariance}. 
Roughly speaking, we show that the SRBM limit is determined by reflection off of the $m$ principal boundary cells, and higher-order reflections off of other boundary cells do not affect the SRBM limit. We believe that our extended SRBM invariance principle could be of independent interest for analyzing other discrete processes with intricate boundary structures.

\begin{theorem}[Invariance principle for SRBM for overdetermined Skorokhod decomposition]\label{thm: main}
	Let $R= [R_1, R_2, \dotsm R_k]$ be a $d \times k$ weak-$\mathcal S$ matrix with respect to a function $f$ such that
	\begin{itemize}
		\item [(1)] $f(j) \not = \emptyset$ for all $j$,
		\item [(2)] for each $1 \le i\le m$, $f(j) = \{i\}$ if and only if $j = i$.
	\end{itemize}
	For each positive integer $n$, let $W^n, X^n \in C([0,\infty), \R^d)$ and $Y^n \in C([0,\infty),\R^k)$ be continuous processes defined on some probability space $(\Omega^n, \mathcal F^n, P^n)$ such that 
	\begin{itemize}
		\item [(i)] $W^n = X^n + RY^n$, 
		\item [(ii)] $W^n(t) \in S$ for all $t \ge 0$, $P^n$-a.s.,
		\item [(iii)] $X^n$ converges in distribution as $n \to \infty$ to a $d$-dimensional Brownian motion with drift $\mathbf 0$, covariance matrix $\Sigma$ and initial distribution $\delta_{\mathbf0}$, 
		\item [(iv)] there are constants $\delta^n \ge 0$ such that $\delta^n \to 0$ as $n \to \infty$ and $P^n$-a.s. for $j = 1,2, \dots, k$, 
		\begin{itemize}
			\item [(a)] $Y_j^n(0) = 0$,
			\item [(b)] $Y_j^n$ is non-decreasing, 
			\item [(c)] $\int^t_0 \mathbf{1}_{F^{\delta^n}_{f(j)}}(W^n(s)) dY_j^n(s) = Y_j^n(t)$ for all $t \ge 0$. Here, $F^\delta_i:=\{x \in S: 0 \le n_i \cdot x-b_i\le \delta\}$ and $F_I^\delta:=\cap_{i \in I}F_i^\delta$ for $\delta>0$, $i \in \{1,2,\dots, m\}$, $\emptyset \neq I \subseteq \{1,2,\dots,m\}$.
		\end{itemize}
		\item [(v)] For each (weak) limit point $(W',X',Y')$ of $\{(W^n, X^n,Y^n)\}_{n=1}^\infty$, $X'$ is a $\mathcal F'_t$-martingale, where $\mathcal F'_t = \sigma\{(W',X',Y')(s): 0\le s\le t\}$ for all $t \ge 0$.
	\end{itemize}
	Then $(W^n, X^n, Y^n) \Rightarrow (W,X,Y)$ as $n\rightarrow\infty$ where $Y_j \equiv 0$ a.s. for each $j>m$ and $W$ is an SRBM associated with $(S; \mathbf 0, \Sigma, \hat R, \delta_0)$ with (standard) Skorokhod decomposition $W = X + \hat R\hat Y$. Here, $\hat R := [R_1, R_2, \dots, R_m]$ and $\hat Y = (Y_1, Y_2, \dots, Y_m)^\top$.  
\end{theorem}

We emphasize that our proof of Theorem \ref{thm: main} follows the original ingenious argument of Williams \cite{williams1995semimartingale} very closely with minor modifications. The key change in our proof is that, for each limit point $(W',X',Y')$ of the sequence of processes $(W^{n},X^{n},Y^{n})$, the $k$-dimensional pushing process $Y'$ vanishes on coordinates corresponding to non-principal boundary cells (see Lem. \ref{lem:pushing_process_high_order_vanish}). Dai and Williams \cite{dai1996existence} obtained a similar result for SRBM that if its ($d \times m$) reflection matrix satisfies $\mathcal{S}$-condition, then the corresponding pushing process at the boundary is essentially inactive at the intersection of two or more principal faces. We follow their argument closely to show that we can ignore the boundary behavior of non-principal boundary cells. 

In Proposition \ref{prop: martin}, we provide a useful sufficient condition for verifying condition (v) in Theorem \ref{thm: main}.

\section{Conjectures and Open Problems}
\label{sec:open_problems}

\vspace{-0.1cm}
In this section, we discuss some conjectures and interesting open problems. 

\vspace{-0.1cm}
\subsection{Higher-order asymptotic for the location of the  balls}

In Corollary  \ref{cor:last_ball}, we have shown that  the $i$th ball in a $d$-ball SBBS at time $t$ is between $(1-\eps)n$ and $(1-\eps)n+O(\sqrt{n})$ in expectation, and for $d=2$, it is precisely at $(1-\eps)n+C\sqrt{n}+o(\sqrt{n})$ for an explicit constant $C$. Can one obtain precise asymptotics for the expected location of the $i$th ball in a $d$-ball system? This requires understanding of the fraction of times that the $i$th ball is part of a $k$-soliton for all $2\le k \le d$. We suspect such fractions will contribute higher-order terms in the expansion as we increase $k$. For instance, the simulation in Figure \ref{fig:sim_SBBS_long_term} shows that solitons of length larger than two can emerge spontaneously and keep accumulating mass, which can carry the $i$th ball at a speed proportional to its length. Even in the 3-ball system, determining the fraction of times the last ball belongs to a 2-soliton appears non-trivial. If we add the leftmost ball to the 2-ball system consisting of the middle and the third ball, it can speed up the middle ball, and hence it will make the last two balls meet more frequently. We suspect this will not affect the asymptotic rate $\sqrt{n}$ and only affect the constant factor.

\vspace{-0.1cm}
\subsection{Speed of the disintegration of large solitons}

While our results consider the asymptotic or long-term behavior of SBBS, the short-term behavior also seems interesting. For instance, the simulation in Figure \ref{fig:sim_SBBS_long_term} indicates that there might be a scaling limit of the trajectory of the rightmost ball until it becomes a singleton. 

\vspace{-0.1cm}
\subsection{Time-dependent error probability in SBBS}

We may allow the error probability $\eps=\eps_{n}$ of the stochastic carrier to depend on the time variable $n$. For instance, decreasing $\eps_{n}$ at a suitable rate may help the solitons in the corresponding SBBS to last longer. In our SRBM limit of SBBS with fixed $\eps$, we essentially only see 2-solitons. Is it possible to scale $\eps_{n}$ appropriately so that we see longer solitons in some scaling limit other than the diffusive scaling? 

\vspace{-0.1cm}
\subsection{Invariant measures} 

Another natural question is to ask for a characterization of the invariant measures for the SBBS, suitably extended to the whole line $\Z$. 
This is a classical question regarding TASEP and PushTASEP. For example, in \cite{liggett1976coupling} it is shown that product Bernoulli measures are the only translationally invariant stationary measures for TASEP. For PushTASEP the invariant measures are studied in \cite{guiol1997resultat,andjel2005long}.  Recently, there have been several works on the same question for BBS \cite{ferrari2018bbs,ferrari2020bbs, croydon2019invariant,croydon2023dynamics}. 

\vspace{-0.1cm}
\subsection{Stochastic carrier that may fail to drop balls}

In SBBS, we only allowed the stochastic carrier to fail to pick up a ball it encounters, but forced it to drop balls at vacant sites whenever it is non-empty. We may also allow the carrier to fail to drop balls at a vacant site even if it is non-empty, at some probability, say $\delta$. While the failure to pick up causes solitons to disintegrate, failure to drop off can make solitons move faster and even integrate. For instance, if we have a 1-soliton followed by a 2-soliton with $\eps=0$, then the shorter (hence slower) soliton falling behind will never catch up to the 2-soliton. However, if the carrier fails to drop the 1-soliton at multiple sites, it could fast-forward and merge with the 2-soliton. Analyzing this ``extended SBBS'' seems very interesting. For instance, what would be the diffusive scaling limit (if any) if $\eps=0$ and $\delta>0$? It remains to be seen whether the overdetermined Skorokhod decomposition and extended SRBM invariance principle we developed in this work remain useful for this case. 

\vspace{-0.1cm}
\subsection{Integrability of SBBS and KPZ universality}

Is SBBS integrable in the sense of having exact formulas for the prelimit system (e.g., \cite{petrov2020pushtasep} for PushTASEP)? Since SBBS interpolates between two types of integrable systems of BBS and PushTASEP, and also based on the close connection between SBBS and PushTASEP in the diffusive scaling, we conjecture that it is so. In addition, as we have mentioned in the introduction,  BBS can be obtained by the $q\rightarrow 0$ limit of the quantum integrable systems (e.g., fusion of the 6-vertex model) \cite{fukuda2000energy,hatayama2001factorization} and also is the ultradiscrete limit of the classical Korteweg-de Vries (KdV) equation \cite{inoue2012integrable}. Is there any existing integrable system and some limiting procedure that yields the SBBS?  Coming from the PushTASEP and TASEP side, it seems natural to ask if SBBS belongs to the Kardar-Parisi-Zhang (KPZ) universality class. For instance, with the flat initial configuration (particles occupying all even integers), does the interface $h(t, x)$ at a fixed site $x$ have fluctuations of order $t^{1/3}$?

\vspace{-0.1cm}
\subsection{Multilevel extension of SBBS}
\label{sec:open_full_process}


Is there some natural extension of the unit-capacity SBBS taking values in the set of Gelfand-Tsetlin patterns similar to the multilevel extension of PushTASEP in \cite{gorin2015limits}? If so, what is the diffusive scaling limit? For $c\ge 2$, we know it cannot be Warren's process due to Remark \ref{rmk:SRBM_recursive}.

\section{SBBS with $\eps\nearrow 1$ is PushTASEP}  
\label{sec:proof_micro}

In this short section, we establish the relation between SBBS and PushTASEP stated in Theorem~\ref{thm: SBBS to PushTASEP}. 

\begin{proof}[\textbf{Proof of Theorem~\ref{thm: SBBS to PushTASEP}}]
	Let $\eta^\eps_s = (\eta^{\eps,(1)}_s, \eta^{\eps,(2)}_s, \dots, \eta^{\eps,(d)}_s)$ be the coin flips at time $s$ for the $d$-ball stochastic box-ball system with capacity $c$ and error probability $\eps$. 
	Define the counting process $Y^\eps$ by $Y^\eps(t) = \sum^{\lfloor \frac{t}{1-\eps}\rfloor-1}_{s=0}\eta^\eps_s$ for $t \ge 0$. Since the random variables $\eta^{\eps,(i)}_s$'s are i.i.d. Bernoulli($1-\eps$), it follows that $Y^\eps$ converges weakly to a $d$-dimensional Poisson process $Y$ as $\eps \uparrow 1$ (See \cite[Example 12.3]{billingsley2013convergence}). 
	
	Now fix an arbitrary sequence $(\eps_n)_{n \in \N} \subset (0,1)$ with $\eps_n \uparrow 1$. By Skorokhod's representation theorem \cite[Theorem~1.8, Chapter~3]{ethier2009markov}, we may assume that $Y^{\eps_n} \to Y$ almost surely as $n \to \infty$. More precisely, on a probability space $(\Omega,\mathcal F, P)$, there exists a null set $\mathcal{N} \in \mathcal F$, $P(\mathcal N) = 0$, such that, for each $\omega \in \Omega \setminus \mathcal N$, $Y^{\eps_n}(\omega) \to Y(\omega)$ in the Skorokhod topology. 
	Equivalently, for such $\omega$ there exists a sequence $\{\lambda_n\} \in \Lambda$, where $\Lambda$ is the set of strictly increasing continuous functions $\lambda:\R_{\ge 0} \to \R_{\ge 0}$ with $\lambda(0) = 0$ and $\lim_{t \to \infty}\lambda(t) = \infty$, such that:  
	\begin{enumerate}
		\item $\lim_{n \to \infty}\sup_{t}|\lambda_n(t)-t| \to 0$,
		\item $\lim_{n \to \infty}\sup_{t \le T}|Y^{\eps_n}(\omega;\lambda_n(t))-Y(\omega;t)| \to 0$ for every $T \in \N$.
	\end{enumerate}
	
	Fix $\omega \in \Omega \setminus \mathcal N$, and let $\{\lambda_n\}_{n \in \N} \in \Lambda$ be the sequence corresponding to this $\omega$. Then, for any $T \in \N$, there exists $N> 0$ such that $Y^{\eps_n}(\omega;\lambda_n(t))=Y(\omega;t)$ for all $0 \le t\le T$ and $n \ge N$. 
	
	Let $D^c([0,\infty),\Z^d_{\ge 0})$ denote the space of $d$-dimensional counting paths, and $D([0,\infty),\Z^d_{\ge 0})$ the space of $\Z^d_{\ge 0}$-valued c\'adl\'ag paths. Define $G: \{0,1\}^d \times \Z^{d-1}_{\ge 0} \to \Z^{d}_{\ge 0}$ so that $G(\eta, \bar x)$ represents the increments of the positions of a $d$ balls, where $\eta$ is the coin flip outcome and $\bar x$ is the vector of gaps between adjacent balls. 
	
	Define $\pi:\mathbb{S}^{d}\to \Z^{d-1}_{\ge 0}$ by $\pi(x) = (x_2-x_1-1, x_3-x_2-1,\dots,x_d-x_{d-1}-1)$, which denotes the gaps between adjacent balls when $x$ refers to a configuration of balls. 
	Define $F: D^c([0,\infty),\Z^d_{\ge 0})\to D([0,\infty),\Z^d_{\ge 0})$ recursively as follows. Letting $\{T_n\}_{n=1}^{\infty}$ denote the jump times of $z,$
	\begin{enumerate}
		\item Set $F^0(z) \equiv \mathbf 0$ for all $z \in D^c([0,\infty),\Z^d_{\ge 0})$.
		\item Inductively, define $F^n(z) = F^{n-1}(z)+\mathbf{1}_{\{t \ge T_n\}}G(\Delta z(T_n), \pi( F^{n-1}(z;T_n)))$.
		\item Define $F(z;t): = \lim_{n \to \infty}F^n(z;t)$ for all $t \ge 0$.
	\end{enumerate}
	By construction, $F$ is non-anticipative. Therefore, for all $0 \le t \le T$ and $n \ge N$, $F(Y^{\eps_n}(\omega);\lambda_n(t))=F(Y(\omega);t)$. Moreover,
	$F(Y^\eps;\cdot) \equiv \zeta^{c,\eps}_{\lfloor \frac{\cdot}{1-\eps} \rfloor}$ for all $\eps$. Hence, $\zeta^{c,\eps_n}_{\lfloor \frac{\cdot}{1-\eps}\rfloor} \Rightarrow F(Y;\cdot)$ as $n \to \infty$, and consequently $\zeta^{\eps} \Rightarrow F(Y)$ as $\eps \uparrow 1$. 
	
	Finally, to identify $F(Y)$, note that $Y$ is a Poisson process, so $Y(T_{n+1})-Y(T_{n})= e_i$ for some $i \in \{1,2,\dots d\}$ a.s. By construction, $F(Y)$ jumps when and only when $Y$ jumps. Suppose the $n$th jump satisfies $Y(T_{n+1})-Y(T_{n})= e_i$ and $F(Y;T_n-) = x$. Then $G(e_i, \pi(x)) = F(Y;T_{n+1})-F(Y;T_n)$, which corresponds to the $i$th leftmost particle moving forward and pushing any particles in front of it. Hence, $F(Y)$ follows the dynamics of the $d$-particle PushTASEP. 
\end{proof}

\section{The gap process for SBBS}
\label{sec:2-ball}

In this section, we analyze the ``gap process'' for the SBBS. Instead of keeping track of the absolute positions of $d$ balls, we can instead track the gaps between them. It would be very useful for the analysis of the SBBS since it will dictate the formation of solitons. In particular, we will prove Theorems \ref{thm:upper_bd_local_time_bdry}, \ref{thm:upper_bd_local_time_bdry_PushTASEP}, and Corollary \ref{cor:last_ball} using the gap process. 

Let $\zeta_{t}=(\zeta_{t}(1),\dots,\zeta_{t}(d))$ denote the SBBS trajectory with $d=2$.  Namely, let 
\begin{align}\label{eq:def_gap_process}
    W_{t}:=(W_{t}^{1},\cdots, W_{t}^{d-1})=\pi(\zeta_{t}):=(\zeta_{t}(2)-\zeta_{t}(1)-1, \cdots, \zeta_{t}(d)-\zeta_{t}(d-1)-1) \in \Z_{\ge 0}^{d-1}.
\end{align}
It is easy to check the \textit{gap process} $(W_{t})_{t\ge 0}$ is a Markov chain on $\Z_{\ge 0}^{d-1}$. The map $\pi$ that computes the gap of a given ball configuration is an affine transformation $x\mapsto N^{\top}x- \mathbf{1}_{d-1}$, 
where $N=[n_{1},\dots,n_{d-1}]$ is the $d\times (d-1)$ matrix of normal vectors for the faces of the Weyl chamber defined below \eqref{eq:def_Weyl_chamber}.

\subsection{The gap process for $d=2$}

Consider SBBS with $d=2$. We analyze the gap between the two balls, $W_{t}:=\zeta_{t}(2)-\zeta_{t}(1) \ge 1$. 
The analysis here will serve as the building blocks for our analysis of the general case. We will assume an arbitrary capacity $c\ge 1$. It is a Markov chain on $\Z_{\ge 0}$ with the following transition kernel: For $d_{0}\ge 1$, 
\begin{align}\label{eq:2balldist}
    \mathbb{P}(W_1=d_1 | W^1_0=d_0) = \begin{cases} 
      (1-\eps)\eps    & \textup{if $d_{1}=d_{0}-1$}\\
    (1-\eps)^{2} + \eps^{2}    & \textup{if $d_{1}=d_{0}$}\\
       (1-\eps)\eps    & \textup{if $d_{1}=d_{0}+1$}
    \end{cases}
\end{align}
and for $d_{0}=0$, 
\begin{align}
    \mathbb{P}(W_1=d_1 | W_0=0)  
    = & \begin{cases} 
      2(1-\eps)+ \epsilon^2 & \text{ if } d_1=0 \\
    (1-\epsilon) \epsilon & \text{ if } d_1=1.
    \end{cases}
\end{align}

We begin with a simple observation that, regardless of the initial configuration, the two balls meet and form a 2-soliton eventually, but the expected time between such events is infinite. 

\begin{proposition}
Suppose $d=2$ and $c\ge 2$. The gap returns to zero a.s., and the expected time to return to zero is infinite. 
\end{proposition}

\begin{proof}
Note that $\mathbb{P}(W_1=i | W_0= j) = \mathbb{P}(W_1=j | W_0=i)$ for all $i,j$. Hence, if the 2 balls started with $W_0=0$, let $T_b$ be the time until they first break up. We have
\begin{align}
    \mathbb{P}(T_b=n) = &  \left(2-2\eps+\eps^2\right)^{n-1} \left(\eps(1-\eps)\right)
\end{align}
with
\begin{align}
    \mathbb{E}[T_b] = \frac{1}{\eps(1-\eps)}.
\end{align}
Exactly analogously to the Gambler's ruin problem for the simple random walk, one can show that if $x\in\{0,\ldots,N\}$ then
\[
\mathbb{P}(\text{hit $N$ before $0$}|W_0=x) = \frac{x}{N}
\]
and
\[
\mathbb{E}[\text{time to hit $0$ or $N$}|W_0=x]= \frac{x(N-x)}{2(1-\eps)\eps}.
\]
The assertion then follows. 
\end{proof}

Let $X=(X_{t})_{t\ge 0}$ be the corresponding bulk process in \eqref{eq:def_bulk_process}. In this case, this is just a lazy simple symmetric random walk on $\Z$ with transition kernel given by \eqref{eq:2balldist}. The following is a useful observation for the 1-dimensional gap process $W_{t}$.

\begin{proposition}\label{prop:d2_W_X}
Suppose $d=2$ and $c\ge 1$. Then $W_{t} = X_{t} - \min_{0\le s \le t} X_{s}$ for $t\ge 0$. 
\end{proposition}

\begin{proof}
    First, suppose $c\ge 2$. 
    We can prove the identity by induction on $t\ge 0$. Denote $M_{t}:=\min_{0\le s \le t}X_{s}$. 
    For $t=0$, $W_{0}=0=X_{0}-X_{0}$ so the assertion holds. Suppose we have $W_{t-1}=X_{t-1} - M_{t-1}$ for some $t\ge 1$. Let $\eta_{t}=(\eta^{(1)}_{t},\eta^{(2)}_{t})$ denote the coin flip vector introduced above \eqref{eq:def_bulk_process}. If $W_{t-1}\ge 1$, then by the induction hypothesis  $X_{t-1}> M_{t-1}$ so  $M_{t-1}=M_{t}$. It follows that 
    \begin{align}
        W_{t}-W_{t-1}=\eta^{(2)}_{t}-\eta^{(1)}_{t} = X_{t}-X_{t-1}. 
    \end{align}
    If $W_{t-1}=0$ and $\eta^{(2)}_{t}-\eta^{(1)}_{t}\ge 0$, then the above still holds. The only remaining case is when $W_{t-1}=0$ and $\eta^{(2)}_{t}-\eta^{(1)}_{t}=-1$, where we have $W_{t}=0$, $X_{t}-X_{t-1}=-1$, and $M_{t}-M_{t-1}=-1$. This completes the induction step. 

    Now consider $d=1$. The only difference is when the two balls form a soliton, and both balls are instructed to be picked up by the coin flips; the second ball will be skipped since the carrier is fully loaded. But the gap still remains 0, which agrees with the infinite capacity case where both balls are picked up. Then the identity can be checked by a similar argument.
\end{proof}

Next, we will prove the following $d=2$ case of Theorem \ref{thm:upper_bd_local_time_bdry} and Corollary \ref{cor:last_ball}. 

\begin{proposition}\label{prop:upper_bd_local_time_bdry_d2}
Suppose $d=2$. Let $N^{(2)}_n$ denote the number of times that $W_{t}=0$ for $0\le t\le n$ and let $\zeta_{n}(2)$ denote the location of the rightmost ball at time $n$. Then \eqref{eq:N_n_2_ball} and \eqref{eq:zeta_2_ball} hold. 
\end{proposition}

For its proof, we will need a basic fact about the local time of a mean-zero random walk at running minimum. 
A related statement is that the number of returns to the origin of a simple random walk on $\Z$ in the first $n$ steps is of $\Theta(\sqrt{n})$ in expectation. This follows easily by writing the expectation as the sum of return probabilities and applying the local central limit theorem. In the following lemma, we establish the same asymptotic for the ``returns to the running minima'', which involves writing the expectation as the sum of ``survival probabilities''. 

\begin{lemma}[Local time of RW at running minima]\label{lem:local_time_RW} 
    Let $S_{n}:=X_{1}+\dots+X_{n}$, $n\ge 0$, $S_{0}=0$ denote a random walk on $\Z$ where the increments $(X_{k})_{k\ge 1}$ are i.i.d., mean zero, finite variance, and integer-valued. Furthermore, assume   $\P(X_{1}\ge -1)=1$ (downward skip-free) and $\P(X_{1}=-1)>0$. Let $M_{n}:=\min_{0 \le k \le n} S_{k}$ 
    denote the running minimum of $S_{n}$ and let $N_{n}$ denote the times that $S_{k}=M_{k}$  for $0\le k \le n$. Then we have 
    \begin{align}
        \E[N_{n}]=(1+o(1))  \, \P(X_{1}=-1) \sqrt{\frac{2\textup{Var}(X_{1}) n}{\pi }} 
    \end{align}
\end{lemma}

\begin{proof}
    The overall idea of the proof is as follows. By using the indicator trick, we first write the expectation as the partial sum of the probabilities $\P(S_{k}=M_{k})$. This is the probability of the event that the sample path of the walk, if traversed backwards in time, stays above the running minimum $M_k$. Hence, looking backwards in time, this is the "survival probability" of a time-reversed random walk staying above the initial level for $k$ steps. Since these survival probabilities scale as $k^{-1/2}$, the result will follow. Below, we will make this sketch rigorous.
    
     Let $\cev{M}_{\ell}$ denote the running maxima of the first $\ell$ partial sums of the time-reversed sequence of increments $X_{k},X_{k-1},\cdots,X_{1}, X_{0}$ with $X_{0}:=0$. Note that 
    \begin{align}
        \cev{M}_{k} &= \max_{0\le \ell \le k}\{ X_{k}+\cdots+X_{k-\ell+1} \} = \max_{0\le \ell \le k}\{ S_{k} - S_{k-\ell} \} 
        = S_{k} - M_{k}.
    \end{align}
    Thus we have  $\P(S_{k}=M_{k}) = \P(\cev{M}_{k}=0)$. Since the increments are stationary, $\cev{M}_{k}$ and $\vec{M}_{k}:=\max_{0\le \ell \le k} S_{\ell}$ have the same distribution for $k\ge 0$. So it follows that 
    \begin{align}\label{eq:EN_n}
       \E[N_{n}] = \sum_{k=0}^{n} \P(S_{k}=M_{k}) = \sum_{k=0}^{n} \P(\vec{M}_{k}=0). 
    \end{align}
    The probabilities in the right-hand side are known as the survival probabilities of the random walk $S_{n}$, that is, the probability that the partial sums $S_{\ell}$ stay non-positive for all $0\le \ell \le k$. Asymptotically they are $\P(\vec{M}_{k}=0) \sim c/\sqrt{k}$ for some explicit constant $c$. More precisely, let $\mathtt{Q}(r,k)$ denote the probability that $\max\{S_{0},\cdots,S_{k}\}\le 0$ with $S_{0}=r\ge 0$. Since $\P(X_{1}\ge -1)=1$, we have $\P(\vec{M}_{k}=0)=\mathtt{Q}(0,k)=\mathtt{Q}(1,k+1) \P(X_{1}=-1)$. Now by \cite[Theorem 2]{lyu2019persistence}, 
    \begin{align}
        \mathtt{Q}(1,k+1) = (1+o(1))\sqrt{\frac{\textup{Var}(X_{1})}{2\pi k}}.
    \end{align}
    The claimed asymptotic for $\E[N_{n}]$ then immediately follows from the above and \eqref{eq:EN_n}. 
\end{proof}

Now we prove Proposition \ref{prop:upper_bd_local_time_bdry_d2}.

\begin{proof}[\textbf{Proof of Proposition \ref{prop:upper_bd_local_time_bdry_d2}}.]
    By Proposition \ref{prop:d2_W_X}, denoting $M_{k}=\min_{0\le j\le k}X_{j}$, 
    \begin{align}
        N^{(2)}_{n}=\sum_{0\le k\le n} \mathbf{1}(X_{k}=M_{k}).
    \end{align}
    Since $X_{k}$ is a simple symmetric lazy random walk on $\Z$ with $\textup{Var}{X_{1}}=\eps(1-\eps)=\P(X_{1}=-1)$,  we can apply Lemma \ref{lem:local_time_RW} to deduce \eqref{eq:N_n_2_ball}. 

    Next, we show \eqref{eq:zeta_2_ball}.  We will first consider the second particle. 
    Observe that for any $k\ge 0$, \begin{align}\label{eq:conditinoal_increment_2soliton}
        \E[ \zeta^{(2)}_{k+1}-\zeta^{(2)}_{k} \,|\, W_{k}=0 ] = \begin{cases}
            1-\eps^{2} & \textup{if $c=1$},\\
             2(1-\eps)
             & \textup{if $c\ge 2$}.
        \end{cases}
    \end{align}
    Indeed, if $c=1$, then the 2-soliton stays put
    If both balls are not picked up by the carrier, which occurs with probability $\eps^{2}$, and moves one step to the right otherwise. If $c\ge 2$, then it moves two sites forward with probability $(1-\eps)^{2}$ (recall that $c\ge 2$), one site forward if exactly one of the two balls is picked up, and stays put otherwise. 
    

    Now by conditioning each increment $\zeta^{(2)}_{k+1}-\zeta^{(2)}_{k}$ on whether the gap $W_{k}$ is positive or not, we have 
    \begin{align}
        \E[\zeta_{n}(2)-\zeta_{0}(2)] &= \sum_{0\le k < n} \E[ \zeta^{(2)}_{k+1}-\zeta^{(2)}_{k} ] \\
        &= \sum_{0\le k < n} \E[ \zeta^{(2)}_{k+1}-\zeta^{(2)}_{k} \,|\, W_{k}\ge 1 ] \, \P(W_{k}\ge 1) +  \E[ \zeta^{(2)}_{k+1}-\zeta^{(2)}_{k} \,|\, W_{k}=0 ] \, \P(W_{k}=0) \\
        &= \sum_{0\le k < n} (1-\eps) \, (1-\P(W_{k}=0)) +   \E[ \zeta^{(2)}_{1}-\zeta_{0}(2) \,|\, W_{0}=0 ] \, \P(W_{k}=0) \\
        &= (1-\eps)n +\left( \E[ \zeta^{(2)}_{1}-\zeta_{0}(2) \,|\, W_{0}=0 ]  - (1-\eps) \right) \E[N^{(2)}_{n}].
    \end{align}
    Simplifying the last expression by using \eqref{eq:conditinoal_increment_2soliton} then shows \eqref{eq:zeta_2_ball} for the second ball.

    Similarly for the first particle, we have \begin{align}\label{eq:conditinoal_increment_2soliton_1st}
        \E[ \zeta^{(1)}_{k+1}-\zeta^{(1)}_{k} \,|\, W_{k}=0 ] = \begin{cases}
            1-\eps & \textup{if $c=1$},\\
             (1-\eps)(2-\eps)
             & \textup{if $c\ge 2$}.
        \end{cases}
    \end{align}
    Then  \eqref{eq:zeta_2_ball} for the first ball follows from the same computation. 
\end{proof}

Lastly, in this section, we will prove some preliminary diffusive limit results for the $d=2$ case of Theorem \ref{thm: SRBM_weak_convergence} by elementary arguments. 

\begin{proposition}\label{prop:SRBM_d2}
    
    Let $(\zeta_{\lfloor t\rfloor}(2)-\zeta_{\lfloor t\rfloor}(1))_{t\ge 0}$ denote the continuous-time gap process for the $2$-ball SBBS with arbitrary capacity $c$ and error probability $\eps\in (0,1)$. Then as $n\rightarrow\infty$, 
			\begin{align}
			(n^{-1/2}W_{\lfloor \frac{nt}{\eps(1-\eps)} \rfloor}\,: \, t\ge 0 ) \Rightarrow |B| \quad \text{as $n \to \infty$},
			\end{align}
			where $B$ denotes the standard Brownian motion in $\R$. 
\end{proposition}

\begin{proof}
    Recall that the bulk process $(X_{t})_{t\ge 0}$ in \eqref{eq:def_bulk_process} for $d=2$ is a simple random walk on $\Z$ with mean zero and variance $\eps(1-\eps)$ increments. By Donsker's theorem, the continuous-time process $(n^{-1/2}X_{\lfloor \frac{nt}{\eps(1-\eps)} \rfloor}\,: \, t\ge 0 )$ converges weakly to the standard Brownian motion. Then, since $W_{t}=X_{t}-\min_{0\le s \le t} X_{s}$ by Proposition \ref{prop:d2_W_X}, the claimed weak convergence of the gap process to the reflecting Brownian motion follows from the L\'evy's $M-B$ theorem (see \cite[Ch.\ 2.3]{morters2010brownian}). 
\end{proof}

Note that we can also view the decomposition $W_{t}=X_{t}-\min_{0\le s \le t} X_{s}$ of the gap process as the standard 1-dimensional Skorokod decomposition by writing $-\min_{0\le s \le t} X_{s}=RY_{t}$ with $R=1$. Clearly, the approach in the proof of Proposition \ref{prop:SRBM_d2} above does not work for $d\ge 3$. 

\subsection{Gap process for general $d\ge 2$}

Next, we prove Theorem \ref{thm:upper_bd_local_time_bdry}. For this, we will need two lemmas. First, in the following lemma we show that an arbitrary initial configuration completely ``disintegrate'' 
into isolated balls in a finite expected time. 

\begin{lemma}[Almost sure disintegration in SBBS]\label{lem:disintegrate}
    Consider an SBBS with capacity $c$, error probability $\eps \in (0,1)$, and number $d$ of balls. Define $\gamma: = \min\{t \in \Z_{\ge 0}: W_t > \mathbf 0\}$, where $W$ is the gap process. Then for any initial state $x \in \Z^{d-1}_{\ge 0}$, $\E[\gamma|W_0 = x]< \infty$. 
\end{lemma}
\begin{proof}
    For any configuration $w \in \Z^{d-1}_{\ge 0}$ of gap process, define $M(w):=\mathbf{1}(w_1=0)+\sum_{i=2}^{d-1}\mathbf{1}(w_{i-1}=0 \ \text{or} \ w_{i}=0) + \mathbf{1}(w_{d-1}=0)$, that is, $M(w)$ tells us number of balls that are not alone in the configuration $w$. We must have $M(w)\in \{0,1,\ldots,d\}$ for all $w \in \Z^{d-1}_{\ge 0}$.

    For every configuration $w$, let $p_{-}(w),p_0(w),p_+(w)$ be the probability that $M(W)$ decreases, stays the same, or increases, resp., after one step of the dynamics when $W$ is at $w$. Note that for every configuration in which $M(w)\ne 0$ we have $p_{-}(w)>\{\min(\eps, 1-\eps)\}^d$.
    More precisely, let $ i_M = \max\{1 \le i \le d-1: w_i = 0\}$. Then when $\eta_t = (0,\dots,0,1, \dots, 1)$ where the first $1$ occurs at the $(i_M+1)$th coordinate, $M(W_t)$ decreases. 

    We define a Markov chain $(\bar M(t), W_t)_{t \in \Z_{\ge 0}}$ on $\{1,2,\dots, d\}\times \Z^{d-1}_{\ge 0}$ so that 
    \begin{align*}
        \text{when }\bar M(t) = 0, \ \bar M(t+1) = d \quad \text{and} \quad \text{when }\bar M(t) >0, \  \bar M(t+1) = \begin{cases}
            \bar M(t)-1 \quad &\text{if}\quad \eta_t = \eta^{W(t)} \\
            d \quad &\text{else}
        \end{cases},
    \end{align*}
where, for each $w$ with, if $M(w) \ne 0$, $\eta^w$ is $(0,\dots,0,1,\dots, 1)$ where the first $1$ occurs at the $(i_M+1)$th coordinate, and if $M(w) = 0$, $\eta^w = (0,\dots, 0)$. Then $\bar M(t) \ge M(W_t)$. If $\bar \gamma :=\min\{t\in \Z_{\ge 0}:\bar M(t)=0\}$, then $\E[\gamma|W_0 = x]\le \E[\bar\gamma|W_0 = x]<\infty$.
\end{proof}

Now we are ready to prove Theorem \ref{thm:upper_bd_local_time_bdry}. Our argument is inductive in nature.

\begin{proof}[\textbf{Proof of Theorem \ref{thm:upper_bd_local_time_bdry}}]

We have already shown the $d=2$ case in Proposition \ref{prop:upper_bd_local_time_bdry_d2}. Hence we only need to show $\E[N^{(d)}_{n}]=\Theta(\sqrt{n})$ for $d\ge 3$. Let $W_{t}=(W^{1}_{t},\cdots,W^{d-1}_{t})$ denote the gap process where $W^{j}_{t}$ denotes the gap between the $j$th and the $j+1$st balls from the left at time $t$.

Our main tool will be a coupling between an adjacent pair of balls $\zeta^{(i)}$ and $\zeta^{(i+1)}$ with a two-ball system (see Sec. \ref{sec:2-ball}) using the same coin flips. For the 2-ball system, Lemma \ref{lem:local_time_RW} shows that the number of times the gap hits zero in $n$ steps is $\sim C\sqrt{n}$ and although knowing this for 2-ball system does not immediately allow us to bound the number of times $W_i$ hits zero in the $d$-ball system, the coupling will allow us to relate the increment of the gap $W_i$ with that of the 2-ball system. Understanding the increments allows us to determine how much the $d$-ball SBBS differs from the 2-ball case.

\textbf{Lower bound.} We will first show the lower bound $\E[N^{(d)}_{n}]=\Omega(\sqrt{n})$ by using the coupling argument described above. Namely, define a new two-ball system with gap process $\tilde W^{d-1}$ 
determined by the coin flips $\tilde\eta = (\eta^{(d-1)}, \eta^{(d)})$, where $\eta^{(d-1)}$ and $\eta^{(d)}$ are the last two coin flips of the original system. By Proposition \ref{prop:upper_bd_local_time_bdry_d2}, we have $\E[\tilde{W}^{d-1}]=\Theta(\sqrt{n})$. 

Note that the coupling yields the following relationship between the dynamics of the two systems, independent of the capacity $c$:
\begin{itemize}
\item If $\tilde W^{d-1}_{t+1}-\tilde W^{d-1}_{t}=-1$, then $W^{d-1}_{t+1}-W^{d-1}_{t}\le -1$, since the presence of the first through the $d-2$nd balls can only accelerate the $d-1$st ball in the $d$-ball system.
\item If $W^{d-1}_{t+1}-W^{d-1}_{t}=1$, then necessarily $\tilde W^{d-1}_{t-1}-\tilde W^{d-1}_{t}=1$, noting that $W_1$ can decrease by at most one in a single step.
\end{itemize}
This yields $W^{d-1}_{t}\le \tilde{W}^{d}_{t}$ for all $t$. Then we deduce 
\begin{align}
\E[N^{(d)}_{n}] \ge \mathbb{E}[\#\{t\le n:W^{d-1}_{t}=0\}] \ge \mathbb{E}[\#\{t\le n:\tilde W^{d-1}_{t}=0\}] = \Omega(\sqrt{n}). 
\end{align}

\vspace{0.1cm}
\textbf{Upper bound.} Now we show the upper bound $\E[N^{(d)}_{n}]=O(\sqrt{n})$. The proof proceeds as follows. For each $k = 1, 2, \dots, d-1$, we show that the expected number of times at which at least one of the first $k$ entries of $W$ is zero up to time $n$ is $O(\sqrt{n})$. The argument is carried out by induction on $k$. 

\vspace{0.1cm}
\textit{Base step.} Consider the case $k=1$. Similarly to the lower bound argument, define a new two-ball system with gap process $\tilde W_1$ 
determined by the coin flips $\tilde\eta = (\eta^{(1)}, \eta^{(2)})$. 
By Proposition \ref{prop:upper_bd_local_time_bdry_d2}, we have $\E[\tilde{W}^{1}]=\Theta(\sqrt{n})$.

As before, the coupling yields the following relationship between the dynamics of the two systems, independent of the capacity $c$:
\begin{itemize}
\item If $\tilde W_1(t+1)-\tilde W_1(t)=1$, then $W_1(t+1)-W_1(t)\ge 1$, since the presence of the third through $d$-th balls can only accelerate the second ball in the $d$-ball system.
\item If $W_1(t+1)-W_1(t)=-1$, then necessarily $\tilde W_1(t+1)-\tilde W_1(t)=-1$, noting that $W_1$ can decrease by at most one in a single step.
\end{itemize}
Consequently, if we set $\tilde W^{1}_{0} = W^{1}_{0}$, then $\tilde W_1(t) \le W_1(t)$ for all $t \ge 0$. Then we get 
\[
\mathbb{E}[\#\{t\le n:W^{1}_{t}=0\}] \le \mathbb{E}[\#\{t\le n:\tilde W^{1}_{t}=0\}] = O(\sqrt{n}),
\]
This completes the base step.

\vspace{0.1cm}
\textit{Induction step.} Now fix $k\in\{2,\ldots,d-1\}$. Suppose we know that
\[
\mathbb{E}[\#\{t\le n:\text{at least one of 
 } W^{1}_{t},\ldots,W^{k-1}_{t} \text{ is zero}\}] = O(\sqrt{n}).
\]
Similarly to the base step, consider a two-ball two-ball system with gap process $\tilde W^{k}$ and bulk process $\tilde X^k$, 
determined by the coin flips $\tilde\eta = (\eta^{(k)}, \eta^{(k+1)})$, where $\eta^{(k)}$ and $\eta^{(k+1)}$ are the $k$th and $(k+1)$th coin flips of the original system.

Suppose we start our processes with $W_i(0)>0$ for each $i=1,\ldots,k-1$, and $W^{k}_{0} = \tilde W^{k}_{0}$. Let $T_1$ be the first time at least one of $W^1,\ldots,W^{k-1}$ hits zero. Before $T_1$, the first $k-1$ balls have no influence on the $k$th ball, and so, just as in the base case, we have
\[
W^{k}_{t} \ge \tilde W^{k}_{t}
\]
for all $t\in [0,T_1)$. The $d$-ball system will then spend some amount of time $\gamma_1$ in which at least one of $W^{1},\ldots,W^{k-1}$ is equal to zero. Note that it follows from the same proof as Lemma \ref{lem:disintegrate} (now with $k$ balls instead of $d$) that $\mathbb{E}[\gamma_1] = C_{1,k} < \infty$, where $C_{1,k}$ is a constant depending on $k$ but not $n$. In fact, we have $\mathbb{E}[\gamma_1] \le \bar C_k$ where $\bar C_k$ is the expectation in the worst-case scenario in which we start with all $k$ balls forming a single soliton. 

Let $T_2$ be the first time after time $T_1+\gamma_1$ that at least one of $W^1,\ldots,W^{k-1}$ hits zero. Note that we are not guaranteed that $W^k_{T_1+\gamma_1} \ge \tilde W^k_{T_1+\gamma_1}$ as during the interval $[T_1,T_1+\gamma_1)$ the 1st through $k$th balls may have formed solitons that causes $W^k$ to decrease faster that $\tilde W^k$. However, it is still true for $t\in [T_1+\gamma_1,T_2)$ that, regardless of capacity, we have 
\begin{itemize}
\item if $\tilde W^k_{t+1}-\tilde W^k_{t} = 1$, then  $ W^k_{t+1}- W^k_{t} \ge 1$,
\item and if $ W^k_{t+1}- W^k_{t} =- 1$ then $\tilde W^k_{t+1}-\tilde W^k_{t} = -1$ as well.
\end{itemize}
That is, the increments of the walk still behave nicely. We see that for any $s,t\in [T_1+\gamma_1,T_2)$, $s<t$ we have
\begin{equation} \label{eq:niceincr}
X^k_{t}-X^k_{s} \ge \tilde X^k_{t} - \tilde X^k(s).
\end{equation}
In particular, if at any time $s\in [T_1+\gamma_1,T_2)$ we have $W^k_{s} \ge \tilde W^k_{s}$, then that will remain true for all $t>s$, $t\in [T_1+\gamma_1,T_2)$. 
Hence, we may write
\begin{equation}
\begin{aligned}
\#\{t\in[T_1+\gamma_1,T_2) \,|\, W^k_{t}=0\} \le &\; \#\{t\in[T_1+\gamma_1,T_2) \,|\, \tilde W^k_{t}=0\} \\
&\hspace{-2cm} + \#\{\text{extra times $W^k$ hit zero before the first time $\tilde W^k\le W^k$ during $[T_{1}+\gamma_{1}, T_{2})$}\}.
\end{aligned}
\end{equation}
We are left to bound these extra hits.

For $t\in [T_1+\gamma_1, T_2)$, define the running minimums
\begin{equation}
    \begin{aligned}
    m_t := \min_{T_1+\gamma_1\le s \le t} (W^k_{s}),\qquad 
    \tilde m_t := \min_{T_1+\gamma_1\le s \le t} (\tilde W^k_{s}).
    \end{aligned}
\end{equation}
Observe that \eqref{eq:niceincr} implies that if $\tilde W^k_{t}> \tilde m_t$ then $ W^k_{t}> m_t$ as well. To see this, let $s$ be the last time $\tilde W^k$ hit its running minimum, we have
\[
\begin{aligned}
    0 < \tilde W^k_{t}-\tilde m_t = \tilde W^k_{t}-\tilde W^k_{s} \le &\;  W^k_{t}- W^k_{s} \qquad \text{by \eqref{eq:niceincr}} \\
    \le &\; W^k_{t}-m_t.
\end{aligned}
\]
Thus, the only times $t$ it is possible for $W^k_{t}=0$ are those for which $\tilde W^k_{t}=\tilde m_t$ as well. 

Note that every time $\tilde W^k$ visits its running minimum, it will stay there for an independent geometric amount of time $\nu_1$ with success parameter $2\eps(1-\eps)$. At the end of this, there is probability $\eps(1-\eps)$ that the minimum decreases, and so it takes a geometric number of visits $\nu_2$ with success parameter $\eps(1-\eps)$ for the minimum to decrease. Finally, note that if $m_t = 0$ then $\tilde m_t \le (k+1) \cdot \gamma_1$ 
Indeed, at time $T_1+\gamma_1$ we have $\tilde W^k(T_1+\gamma_1) - W^k(T_1+\gamma_1) \le k \cdot\gamma_1$ since $W_{k}$ increase by at most $k-1$ when the first $k-1$ balls form a soliton while $\tilde{W}_{k}$ increase by at most $1$. Since $\tilde{W}_{k}$ and $W_{k}$ behaves exactly the same during $[0,T_{1}]$, the claimed inequality follows. Now this difference $\tilde W^k(T_1+\gamma_1) - W^k(T_1+\gamma_1)$ can only shrink during the interval $[T_1+\gamma_1,T_2)$. Putting things together, we have
\[
\#\{\text{extra times $W^k$ hit zero before the first time $\tilde W^k\le X^k$ during $[T_{1}+\gamma_{1}, T_{2})$}\} \le (k-1)\cdot\gamma_1\cdot\nu_1\cdot\nu_2.
\]

Now consider running the processes for $n$ steps. Let $T_1,\gamma_1$ be defined as above, and recursively define
\[
\begin{aligned}
T_i =& \text{the first time after $T_{i-1}+\gamma_{i-1}$ for which at least one of $W_1,\ldots,W_{k-1}$  is zero}, \\
\gamma_i =& \text{the number of steps starting from time $T_i$ before none of $W_1,\ldots,W_{k-1}$ are zero}
\end{aligned}
\]
We can decompose our steps into 
\[
[0,n] = [0,T_1) \cup [T_1,T_1+\gamma_1) \cup \ldots \cup [T_{m-1}+\gamma_{m-1},T_m) \cup [T_m,T_m+\gamma_{m})
\]
where $m$ is the random number of excursions during $[0,n]$ for which none of $W_1,\ldots,W_{k-1}$ were zero. By the preceding discussion, we may write
\[
\begin{aligned}
    \#\{t\in[0,n]|W^k_{t}=0\} \le &\; \sum_{i=1}^m \#\{t\in[T_{i-1}+\gamma_{i-1},T_i)|W^k_{t}=0\} + \sum_{i=1}^{m} \gamma_i \\
    \le &\; \sum_{i=1}^m \left(\#\{t\in[T_{i-1}+\gamma_{i-1},T_i)|\tilde W^k_{t}=0\} + \#\{\text{extra hits in $[T_{i-1}+\gamma_{i-1},T_i)$}\} \right) + \sum_{i=1}^{m} \gamma_i  \\
    \le &\; \#\{t\in[0,n]|\tilde W^k_{t}=0\} + \sum_{i=1}^m \#\{\text{extra hits in $[T_{i-1}+\gamma_{i-1},T_i)$}\} + \sum_{i=1}^{m} \gamma_i \\
    \le &\; \#\{t\in[0,n]|\tilde W^k_{t}=0\} + \sum_{i=1}^m (k-1)\cdot \gamma_i \cdot \nu_1 \cdot \nu_2 + \sum_{i=1}^{m} \gamma_i
\end{aligned}
\]

We would like to take the expectation of both sides and use Wald's inequality to handle the random sums. To do this, we must show that.
\[
\mathbb{E}[\gamma_i \mathbf{1}(m\ge i)] = \mathbb{E}[\gamma_i]\, \mathbb{P}(m\ge i).
\]
But note that the event $\{m\ge i\}$ depends only on coin flips up to time $T_i$ while $\gamma_i$ only depends on coin flips after $T_i$, so in fact they are independent. 

Now, taking expectations gives
\[
\begin{aligned}
\mathbb{E}[\#\{t\in[0,n]|W^k_{t}=0\}] \le &\;  \mathbb{E}[\#\{t\in[0,n]|\tilde W^k_{t}=0\}] + \mathbb{E}[\sum_{i=1}^m (k-1)\cdot \gamma_i \cdot \nu_1 \cdot \nu_2] +  \mathbb{E}[\sum_{i=1}^{m} \gamma_i] \\
=&\;\mathbb{E}[\#\{t\in[0,n]|\tilde W^k_{t}=0\}] + (k-1)\mathbb{E}[\sum_{i=1}^m \mathbb{E}[\gamma_i \cdot \nu_1 \cdot \nu_2]] +  \mathbb{E}[\sum_{i=1}^m \mathbb{E}[\gamma_i]] \\
\le&\;\mathbb{E}[\#\{t\in[0,n]|\tilde X^k_{t}=0\}] + (k-1)\mathbb{E}[m] \bar C_k \cdot C +  \mathbb{E}[m] \bar C_k \\
=&\;O(\sqrt{n})
\end{aligned}
\]
where as in the base case $\mathbb{E}[\#\{t\in[0,n]|\tilde W^k_{t}=0\}]=O(\sqrt{n})$ by  Proposition \ref{prop:upper_bd_local_time_bdry_d2},
by induction hypothesis $\mathbb{E}[m] = O(\sqrt{n})$, and the expectation of the geometric random variables are all $O(1)$.

To finish the induction step, we note
\[
\begin{aligned}
\mathbb{E}&[\#\{t\le n:\text{at least one of 
 } W^1_{t},\ldots,W^{k-1}_{t},W^{k}_{t} \text{ is zero}\}] \\
 &\le \mathbb{E}[\#\{t\le n:\text{at least one of 
 } W^1_{t},\ldots,W^{k-1}_{t} \text{ is zero}\}] + \mathbb{E}[\#\{t\le n: W^{k}_{t}=0\}] \\
 &=O(\sqrt{n}).
\end{aligned}
\]
This completes the proof.
\end{proof}


\begin{proof}[\textbf{Proof of Corollary \ref{cor:last_ball}}]
    We have already shown the $d=2$ case in Proposition \ref{prop:upper_bd_local_time_bdry_d2}. We can deduce \eqref{eq:last_ball_gen} immediately from \eqref{eq:N_n_2_ball} by noting that the expected speed of the $i$th ball when isolated is $1-\eps$ and otherwise it is between $0$ and $d$. 
\end{proof}

The analogous result for PushTASEP stated in Theorem  \ref{thm:upper_bd_local_time_bdry_PushTASEP} can be shown similarly. 

\begin{proof}[\textbf{Proof of Theorem  \ref{thm:upper_bd_local_time_bdry_PushTASEP}}]

The proof here is similar to that of Theorem  \ref{thm:upper_bd_local_time_bdry}. Here we give a sketch of the argument. 
One can couple a gap in the PushTASEP gap process to a gap in a 2-ball system by using the same exponential clocks for both systems. Using this one has the same relationships between the increments of the two systems as we see in the proof of Theorem  \ref{thm:upper_bd_local_time_bdry} from which one can bound the difference between the local time at the boundary between the two systems. The time the 2-ball system spends at the boundary can be shown to equal the time the continuous-time simple random walk spends at its running minimum, namely, $O(\sqrt{T})$ in expectation. From this, one can deduce \eqref{eq:PushTASEP_rightmost} on the location of the rightmost particle similarly as in the proof of Corollary \ref{cor:last_ball}.
\end{proof}

\section{Proof of diffusive scaling behaviors}
\label{sec: proof of SBBS diffusion limit}

In this section, we prove the results on the diffusive scaling limit of SBBS and PushTASEP stated in Subsection \ref{sec: diffusive scaling gap process}. We do so by invoking the extended SRBM invariance principle in Theorem~\ref{thm: main}, which will be proved in the following section. Recall that $\mathbb{S}^{d}$ denotes the set of all $d$-ball configurations. Define its boundary $\partial \mathbb{S}^{d}$ as 
\begin{align}\label{eq:def_discrete_Weyl_chamber}
   \partial \mathbb{S}^{d} := \{x \in \mathbb{S}^{d}: x_i + 1 = x_{i+1} \text{ for some $i = 1,2,\dots, d-1$}\}, 
\end{align}
which is the set of $d$-ball configurations where at least two balls are next to each other.

We start by taking a closer look at the $d=3$ example. 

\subsection{SBBS with $d=3$ and $c\ge 3$}\label{ex:3ballIncrements}

    Consider the $3$-ball SBBS $\zeta=(\zeta_{t})_{t\ge 0}$ with error probability $\eps \in (0,1)$ and capacity $c\ge 3$. Let $X=(X_{t})_{t\ge 0}$ and $W=(W_{t})_{t\ge 0}$ denote the associated bulk process (see \eqref{eq:def_bulk_process}) and the gap process (see \eqref{eq:def_gap_process}). The dynamics is generated by a sequence of binary vectors $\eta_{t}=(\eta^{(1)}_{t},\eta^{(2)}_{t},\eta^{(3)}_{t})$ of coin flips introduced above \eqref{eq:def_bulk_process}, 
    With $0$ corresponding to a particle being skipped and $1$ corresponding to a particle being picked up, and any sequence not shown results in the walk staying in the same place. For example, when $\zeta_{t}$ lies in the `interior' $\mathbb{S}^{3}\setminus \partial \mathbb{S}^{3}$ where all balls are separated by distance at least two (pictorially $\zeta_{t}=(\bullet \quad \bullet \quad \bullet)$), the coin flip sequence $011$ results in $\delta \zeta_{t} :=\zeta_{t+1}-\zeta_{t}=(0,1,1)$. However, when $\zeta_{t}$ forms a 3-soliton (i.e., $\zeta_{t}(2)-\zeta_{t}(1)=\zeta_{t}(3)-\zeta_{t}(2)=1$, and pictorially $\zeta_{t}=(\bullet \bullet \bullet)$), then the same sequence results in $\delta \zeta_{t} =(0,2,2)$. The boundary $\partial \mathbb{S}^{3}$ decomposes into the following four cells in which the transition kernel for $\zeta$ is homogeneous:
    \begin{align}
        p_{1} &:= \{ \bullet \bullet \qquad \bullet \} =  \{x\in \Z^{3}\,:\, x_{2}-x_{1}=1,\, x_{3}-x_{2}\ge 3 \},  \\
        p_{2} &:=\{ \bullet \quad \bullet  \bullet \} = \{x\in \Z^{3}\,:\, x_{2}-x_{1}\ge 2,\, x_{3}-x_{2}=1 \}, \\
        p_{3} &:=\{ \bullet \bullet  \bullet \} = \{x\in \Z^{3}\,:\, x_{2}-x_{1}=x_{3}-x_{2}=1 \}, \\
        p_{4} &:=\{ \bullet \bullet \quad \bullet \} = \{x\in \Z^{3}\,:\, x_{2}-x_{1}=1,\, x_{3}-x_{2}=2 \}. 
    \end{align}
    The distinction between $p_{1}$ and $p_{4}$ is characteristic of the solitonic interaction in SBBS; the evolution depends on whether the 2-soliton can interfere with the 1-soliton in front of it in one step ($p_{4}$) or not ($p_{1}$). 
    
    Since it is cumbersome to visualize the state-space diagram for $\zeta$ in $\Z^{3}$, we instead consider the associated gap process $W$. The transition kernel for $W$  is shown in Figure \ref{fig:3ball} below. The four yellow shaded regions show the images of the boundary cells $p_{i}$ under the affine transformation $\pi$ in \eqref{eq:def_gap_process}. 
    \begin{figure}[h]
        \includegraphics[width=0.8\textwidth]{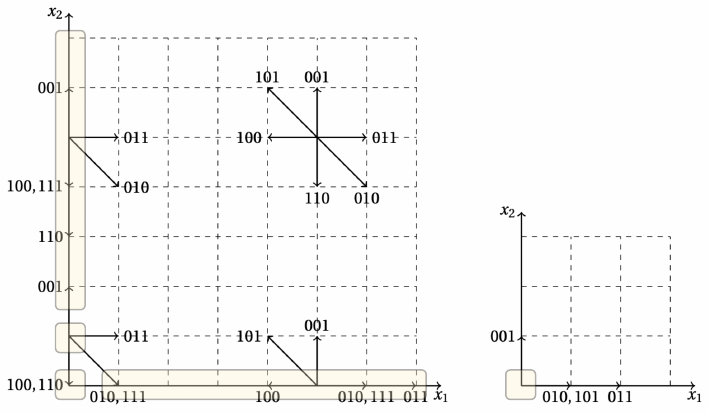}
        \caption{Transition kernel for the gap process of the infinite-capacity 3-ball SBBS. The transition kernel at the origin is shown in the picture on the right. Four boundary cells are indicated by the regions shaded in yellow.}
        \label{fig:3ball}
    \end{figure}

    In order to motivate the overdetermined Skorokhod decomposition for $\zeta$, we now look at the following correction process  
\begin{align}\label{eq:d3_correction}
    \Delta \zeta_t - \Delta X_t = \sum_{j = 1}^4 \mathbf 1(\zeta_t \in p_j)(\Delta \zeta_t - \Delta X_t), 
\end{align}
where $\Delta \zeta_{t} = \zeta_{t+1}-\zeta_{t}$ and so on. Even if we know which boundary cell $p_{j}$  that $\zeta_{t}$ currently belongs to, the actual correction vector $\mathbf 1(\zeta_t \in p_j)(\Delta \zeta_t - \Delta X_t)$ depends on the coin flips $\eta_{t}$. A key idea for the (overdetermined) Skorokhod decomposition is to take the expected correction vector and write out the fluctuation around it. Namely, for each $j \in \{1,2,3,4\}$, we define the mean reflection vector $R_j: = \E[\Delta \zeta_t - \Delta X_t|\zeta_t \in p_j] = \E[\Delta \zeta_t| \zeta_t \in p_j]$. A straightforward computation shows 
\begin{align}\label{eq: d=3reflection}
    R_1 = (1-\eps)
    \begin{bmatrix}
    1-\eps \\
    1 \\
    0
    \end{bmatrix}, \ 
    R_2 = (1-\eps)
    \begin{bmatrix}
    0 \\
    1-\eps \\
    1
    \end{bmatrix}, \ 
    R_3 = (1-\eps)
    \begin{bmatrix}
    2(1-\eps) \\
    2-\eps^{2} \\
    2 
    \end{bmatrix}
    , \ \text{and} \ 
    R_4 = (1-\eps)
     \begin{bmatrix}
    1-\eps \\
    2-2\eps+\eps^{2} \\
    1-\eps
    \end{bmatrix}
    .
\end{align}
Collect the above four columns into a $3\times 4$ mean reflection matrix  $R = [R_1, R_2 ,R_3, R_4]$. Also let $Y^{j}_{t}=\sum_{s=0}^{t-1} \mathbf{1}(\zeta_{s}\in p_{j})$ denote the local time of $(\zeta_{s})_{0\le s \le t}$ on the boundary cell $p_{j}$ and denote $Y_{t} = (Y_{t}^1, Y_{t}^2, Y_{t}^3, Y_{t}^4)^\top$. Then we can further decompose \eqref{eq:d3_correction} as 
\begin{align}\label{eq: increment}
    \Delta \zeta_t = \Delta X_t + R \Delta Y_t + \underbrace{\sum_{j = 1}^4 \mathbf 1(\zeta_t \in p_j)(\Delta \zeta_t - \Delta X_t-R_j \Delta Y^j_t)}_{=:\delta_{t}}. 
\end{align}
By integrating both sides, we obtain the following ``overdetermined'' Skorokhod decomposition: 
\begin{align}\label{eq:d3_decomposition}
    \zeta_t = X_t +  RY_t + \alpha_t. 
\end{align}
where $\alpha_t = \sum^{t-1}_{s=0}\delta_s$. To see the overdetermined structure more clearly, we may write out the components:
\begin{align*}
    \begin{bmatrix} 
    \zeta_{t}(1) \\
    \zeta_{t}(2) \\
    \zeta_{t}(3) 
    \end{bmatrix}  = 
    \begin{bmatrix} 
    X^{1}_{t} \\
    X^{2}_{t} \\
    X^{3}_{t} 
    \end{bmatrix}
    +  [R_{1},R_{2},R_{3},R_{4}] 
    \begin{bmatrix} 
    Y^{1}_{t} \\
    Y^{2}_{t} \\
    Y^{3}_{t} \\
    Y^{4}_{t} 
    \end{bmatrix} 
    + \alpha_t.
\end{align*}
While the process $\zeta_{t}$ is 3-dimensional, there are four boundary cells, so the pushing process $Y_{t}$ is four-dimensional. 

It would be instructive to state the $d=3$ case of Theorem \ref{thm: SRBM_weak_convergence}. Essentially, only the first two columns of $R$ survive in the diffusive limit. 

	\begin{corollary}\label{cor:SRBM_d3}
		Let $(\zeta_{\lfloor t\rfloor})_{t\ge 0}$ denote the continuous-time SBBS with $d=3$, capacity $c\ge 3$, and  error probability $\eps\in (0,1)$. As $n\rightarrow\infty$, 
			\begin{align}
			(n^{-1/2}(\zeta_{\lfloor \frac{nt}{1-\eps} \rfloor}-nt \mathbf{1}_{3})\,: \, t\ge 0 ) \Rightarrow \mathcal{W} \quad \text{as $n \to \infty$},
			\end{align}
			where $\mathcal{W}$ is an SRBM associated with data $(S, \mathbf{0}, \eps I_{3}, \hat{R}, \delta_{\mathbf{0}})$ with $S = \{x \in \R^3: x_1 \le x_2 \le x_{3} \}$  the 3-dimensional Weyl chamber, 
            \begin{align}
                \hat{R}_{\eps} = \frac{1}{1-\eps}[R_{1},R_{2}]=
                \begin{bmatrix}
                   1-\eps & 0 \\
                   1 & 1-\eps \\
                   0 & 1 
                \end{bmatrix}
                .
            \end{align}
	\end{corollary}

    We will prove the full statement in Theorem \ref{thm: SRBM_weak_convergence} by invoking Theorem \ref{thm: main} later in this section, which will immediately imply Corollary \ref{cor:SRBM_d3} above. By the uniformly bounded difference between the floor process $(\zeta_{\lfloor t\rfloor}:t \ge 0)$ and the linearly interpolated process $(\overline \zeta_{t}:t \ge 0)$, it suffices to apply Theorem~\ref{thm: main} to $(\zeta^n, X^n, Y^n)_{n=1}^\infty$, where  
    \begin{align}
    \begin{split}\label{eq: approx seq}
        \zeta^n:=(n^{-1/2}(\overline\zeta_{nt}-nt(1-\eps) \mathbb{1}_{d})\,: \, t\ge 0), X^n:=(n^{-1/2}(\overline X_{nt}-nt(1-\eps) \mathbb{1}_{d} + \overline \alpha_{nt})\,: \, t\ge 0),\text{ and} \\ 
        Y^n:=(n^{-1/2}(\overline Y{nt}-nt(1-\eps) \mathbb{1}_{d})\,: \, t\ge 0 )\quad (\overline\zeta, \overline X, \overline Y, \overline\alpha \text{ are the linear interpolations of }\zeta, X, Y, \alpha, \text{respectively}),
    \end{split}
    \end{align}
    and then divide the resulting covariance matrix and reflection matrix by $1-\eps$.
    Here we will check some of the hypotheses of Theorem \ref{thm: main} for $(\zeta^n, X^n, Y^n)$ when $d=3$ to provide a helpful example.  

    First, we check that the rectangular reflection matrix $R$ above is indeed weakly $\mathcal{S}$ with following boundary map $f$ (see Def. \ref{def:weakly_completely_S}): $f(1) = \{1\}$, $f(2) = \{2\}$, $f(3) = \{1,2\}$, $f(4) = \{1,2\}$. For instance, $\pi(p_{1})$ has degenerate coordinate 1 and both coordinates are degenerate for $\pi(p_{3})$ and $\pi(p_{4})$. Now to check if $R$ is weak $\mathcal{S}$  w.r.t. $f$, we need to show that the submatrices $N^IR^{I}$ are such that some nonnegative linear combination of their rows is a strictly positive vector for all nonempty subsets $I\subseteq \{1,2\}$. Indeed, first note that $N^{\{1\}}R^{\{1\}} = n_1 \cdot R_1 =
        (1-\eps)\eps= N^{\{2\}}R^{\{2\}} = n_2 \cdot R_2=N^{\{2\}}R^{\{2\}}
    $, so if $I=\{1\}$ or $\{2\}$, then the condition holds whenever $\eps\in (0,1)$. The remaining case to check is when $I=\{1,2\}$, for which $N^{\{1,2\}} = N$ and $R^{\{1,2\}} = R$. From \eqref{eq: d=3reflection}, the matrix $NR$ can be expressed in the following form:
    \begin{align*}
        \begin{bmatrix}
            + & + & +  & + \\
            - & + & +  & -
        \end{bmatrix},
    \end{align*}
    where each “$+$” denotes a positive entry and each “$-$” denotes a negative entry. Thus, adding a large enough positive scalar multiple of the first row to the second row gives a strictly positive vector. Hence $R$ is weak $\mathcal{S}$  w.r.t. $f$.

    Next, we check condition (iv)(c) of Theorem \ref{thm: main}. Indeed, note that for each $j \in \{1,2,3,4\}$, $\Delta Y^{j}=0$ if $n_i \cdot \zeta\ge d$ for some $i \in f(j)$. Specifically, $\Delta Y^{1}=0$ if $n_1 \cdot \zeta = W^1+1\ge 3$, $\Delta Y^{2}=0$ if $n_2 \cdot \zeta = W^{2}+1\ge 3$, and $\Delta Y^{j}=0$ if $n_1 \cdot \zeta\ge 3$ or $n_2 \cdot \zeta\ge 3$ for $j\in \{3,4\}$. Since $|\overline \zeta(i) -\zeta(i)|\le 3$, we can set $\delta^n$ to be $\frac{9}{\sqrt n}$.

    Lastly, we comment on how to verify condition (iii) of Theorem \ref{thm: main}. From the decomposition in \eqref{eq:d3_decomposition}, it amounts to showing that the ``error process'' $\alpha_{t}$ is negligible in diffusive scaling. To see why, consider the contribution to $\alpha_{n}$ from a fixed boundary cell. It is the total fluctuation of the random reflections off of that cell during the time interval $[0,n]$. Since all jumps are uniformly bounded, it will be of order the square root of the number of times SBBS hits that boundary cell. Since by Theorem \ref{thm:upper_bd_local_time_bdry} the gap process does not spend more than $\sqrt{n}$ times on \textit{any} boundary, it will be of order $n^{1/4}$. Since the number of total boundary cells may depend on $d$ but not on $n$, this will yield that $\alpha_{n}$ is of order $n^{1/4}$. A more precise statement for the general case is given in Proposition \ref{prop: error}.

\subsection{Constructing the overdetermined Skorokhod decomposition for the general case} 

Continuing the overdetermined Skorokhod decomposition of the SBBS for the $d=3$ case in \eqref{eq: increment}, we introduce the corresponding decomposition for general $d$. Let $\zeta$ be the SBBS with $d$ balls, error probability $\eps \in (0,1)$, and capacity $c \in \N \cup\{\infty\}$, and $X$ be the corresponding bulk process. By denoting the correction process $\zeta-X$ by $A$, we have 
\begin{align*}
    \Delta \zeta_t = \Delta X_t + \Delta A_t = \Delta X_t + \mathbf{1}(\zeta_t \in \partial \mathbb{S}^{d})\Delta A_t.
\end{align*}
As in Example~\ref{ex:3ballIncrements}, the boundary 
$\partial \mathbb{S}^{d}$ of the  state space $\mathbb S^{d} $ of $\zeta$ can be partitioned into finitely many cells, where the $\zeta$, hence the correction process $A$, exhibits distinct behaviors under the coin flips $\eta = (\eta^{(1)}, \eta^{(2)}, \dots, \eta^{(d)})$. A precise description of the boundary cells and the transition kernel for $\zeta$ is stated in the lemma below. We relegate its proof to the end of this section. 

\begin{lemma}\label{lem: component}
    There exists a partition $\{p_1, p_2, \dots, p_k\}$ of $\partial \mathbb{S}^{d}$, which is independent of both $\eps$ and $c$, such that the transition kernel of $\zeta$ is homogeneous within each cell of the partition, and the following properties hold: 
    \begin{enumerate}
        \item For each $j \in \{1,2,\dots, k\}$, there exist $x^j \in \partial\Z^{d-1}_{\ge 1}$ and a nonempty subset $I^j \subseteq \{1,2,\dots, d-1\}$ such that $x^j_i = 1$ for $i < \min I^j$, $x^j_{\min I^j}=0$, and 
        $$p_j = \{y \in \partial \mathbb{S}^{d}: \pi(y)_i = x^j_i \text{ for }i\in I^j, \ \pi(y)_i \ge x^j_i\text{ for }i\in (I^j)^c\},$$
        where $\pi(y) = (y_2-y_1-1, y_3-y_2-1, \dots, y_d-y_{d-1}-1)$.
        \item For each $j \in \{1,2,\dots, d-1\}$, 
        $$p_j = \{y \in \partial \mathbb{S}^{d}: \pi(y)_j = 0, \ \pi(y)_{j+1} \ge 2, \ \pi(y)_i \ge 1\text{ for }i \not \in \{j,j+1\} \}.$$ 
        \item Moreover, when $c \ge d$, the transition kernel for $   \zeta$ is mutually distinct across different cells.  
    \end{enumerate}
\end{lemma}

In particular, part (ii) of the above lemma shows that there are $d-1$ ``principal'' boundary cells, with a single degenerate coordinate for each $j=1,\dots d-1$. By permuting the indices, we will assume that the first $d-1$ boundary cells, $p_{1},\dots,p_{d-1}$, are such principal ones. In terms of the map $f$ that sends a boundary cell index $j$ to the set of degenerate coordinates of $\pi(p_j)$, we have $f^{-1}(\{j\}) = \{j\}$ for $j=1,\dots,d-1$.

Mow for each $j \in \{1,2,\dots, k\}$, we can define 
\begin{align}\label{eq:reflection_def}
R_j := \E[\Delta A_t|\zeta_t = x] \quad \textup{for $t \in \Z_{\ge 0}$ and an arbitrarily chosen $x \in p_j$}.
\end{align}
Note that since $\E[\Delta X_{t}\,|\, \zeta_{t}]=(1-\eps) \mathbf 1_d$, where $\mathbf 1_d$ denotes a $d$-dimensional vector with entries $1$, we also have $R_j := \E[\Delta \zeta_t|\zeta_t = x]-(1-\eps)\mathbb{1}_d$. Then we have
\begin{align}\label{eq: sbbs}
    \Delta \zeta_t = \Delta X_t + \Delta A_t = \Delta X_t + \sum_{j=1}^k \mathbf{1}(\zeta_t \in p_j)\Delta A_t = \Delta X_t + \sum_{j=1}^k R_j\Delta Y^j_t+ \sum_{j=1}^k \mathbf{1}(\zeta_t \in p_j)(\Delta A_t-R_j),
\end{align}
where for each $j \in \{1,2,\dots, k\}$, $\Delta Y^j_t = \mathbf 1 (\zeta_t \in p_j)$, which gives the following \textit{overdetermined} Skorokhod-type decomposition: 
\begin{align*}
    \zeta_t = X_t + RY_t + \alpha_t,
\end{align*}
where $R = [R_1, \dots, R_k] \in \R^{d \times k}$, $Y_t = (Y^1_t, \dots, Y^k_t)^\top$ with $Y_0 = \mathbf 0,$ $Y^j_t = \sum^{t-1}_{s = 0} \Delta Y^j_s$, $\alpha_0 = \mathbf 0,$ and $\alpha_t = \sum^{t-1}_{s=0}\delta_s$, where $\delta_t$ denotes the last term in \eqref{eq: sbbs}. 

With this decomposition complete, we may define our diffusive scaling.
For each $n \in \N$, define $\zeta^n, X^n \in C([0,\infty), \R^{d})$ and $Y^n \in C([0,\infty), \R^k)$ by \eqref{eq: approx seq}. 

By Lemma~\ref{lem: component}, $f$ satisfies the two conditions in Theorem~\ref{thm: main}. 
With $\zeta^n$, $X^n$, $Y^n$, $R$, and $f$ in hand, we aim to verify conditions (i)--(v) for 
\begin{align*}
    S := \{x \in \R^d: x_1 \le x_2 \dots \le x_d\} = \{x \in \R^d: n_i \cdot x \ge 0 \text{ for all $i = 1,2,\dots, d-1$}
    \},
\end{align*}
where $n_i : = (0,\dots,0,-1,1,0,\dots,0)$ whose $i$th coordinate is $-1$, 
and that $R$ satisfies weak $\mathcal S$--condition with respect to $S$ and $f$.
We show that $X^n$ converges to a Brownian motion with covariance matrix $\eps(1-\eps)I_d$, not $\eps I_d$, owing to the time scaling $\frac{t}{1-\eps}$ in \eqref{eq: time scale}. Since $|\overline{\zeta}_t-\zeta_{\lfloor t \rfloor}|$ is uniformly bounded, this process yields the desired result. 
    
By construction, conditions (i), (ii), and (a) and (b) of (iv) immediately follow. 
By Lemma~\ref{lem: component} and the definition of $f$, 
$$p_j \subseteq \{x \in \partial\mathbb{S}^{d}: \pi(x)_i < d-1 \ \text{for} \ i \in f(j) \} = \{x \in \partial\mathbb{S}^{d}: n_i \cdot x < d \ \text{for} \ i \in f(j) \}$$
for all $j \in \{1,2,\dots, k\}$. Additionally, note that for each $i \in \{1,2,\dots,d-1\}$, $|\Delta \zeta(i+1)-\Delta \zeta(i)|$ is uniformly bounded by $d-1$.  
By setting $\delta_n = n^{-1/2}d$, we then have (c) of (iv).
Thus, the proof reduces to checking conditions (iii) and (v) in Theorem \ref{thm: main}. 

\subsection{The rectangular reflection matrix}

We will first compute the principal mean reflection direction $R_{j}$ in \eqref{eq:reflection_def} explicitly. 

\begin{proposition}
    For $j=1,\dots,d-1$, $R_{j}$ in \eqref{eq:reflection_def} is given by 
    \begin{align}\label{eq:SBBS_R_formula}
    R_{j}  =
    \begin{cases}
    (1-\eps)(0,\dots, 0, \eps,0,\dots, 0)^{\top}\quad &\text{for $c=1$},\\
    (1-\eps)(0,\dots, 0, (1-\eps), 1,0,\dots, 0)^{\top}  \quad &\text{for $c\ge 2$},
\end{cases}
\end{align}
where $\eps$ and $1$ appears in the $j+1$st coordinate, respectively. In particular, $(1-\eps)\hat R^{c,\eps} = [R_1,R_2,\dots, R_{d-1}]$, where $\hat{R}^{c,\eps}$ is in \eqref{eq:sigma_R}.
\end{proposition}

\begin{proof}
    We will first argue for $c\ge 2$. 
    Suppose $\zeta_{t}=(\zeta_{t}(1),\dots,\zeta_{t}(d))\in p_{j}$. By (ii) of Lemma~\ref{lem: component}, $\zeta_{t}(j+1) - \zeta_t(j)=1$, $\zeta_{t}(j+2) - \zeta_{t}(j+1) \ge 3$, and $\zeta_{t}(i+1)-\zeta_{t}(i) \ge 2$ for all $i \not \in \{j,j+1\}$. Clearly for each $i \not \in \{j,j+1\}$, $\E[\Delta \zeta_{t}(i)\,|\, \zeta_{t}\in p_{j}]=1-\eps$ since the $i$th ball is isolated at time $t$. Next, consider the two balls forming the 2-soliton at time $t$. The carrier must be empty when it interacts with the $j$th ball.
    The rightmost one of the two changes by $\Delta \zeta_{t}(j+1)$, which is $2$ if both balls are picked up with probability $(1-\eps)^{2}$ (here we use $c\ge 2$), and 1 if exactly one of the two balls are picked up with probability $\eps(1-\eps)$, and 0 otherwise. Hence $\E[\Delta \zeta_{t}(j+1)\,|\zeta_{t}\in p_{j}]=2(1-\eps)^{2}+2\eps(1-\eps)=2(1-\eps)$. Similarly, $\E[ \Delta \zeta_{t}(j)\,|\, \zeta_{t}\in p_{j} ]= 2(1-\eps)^{2}+\eps(1-\eps)=(1-\eps)(2-\eps)$, which gives 
    \begin{align}
        \E[\Delta A_t|\zeta_t \in p_j]=\E[\Delta \zeta_t|\zeta_t\in p_j]-(1-\eps)\mathbb{1}_d = R_j.
    \end{align}

    Now assume $c=1$. The argument is similar as before, with the key differences being $\E[ \Delta \zeta_{t}(j+1)\,|\, \zeta_{t}\in p_{j} ]=1-\eps^{2}$ since the $j+1$st ball moves by 1 unless both the $j$th and the $j+1$st balls are skipped, and $\E[ \Delta \zeta_{t}(j)\,|\, \zeta_{t}\in p_{j} ]=1-\eps$. Hence, we also have \eqref{eq:SBBS_R_formula} for $c=1$.
\end{proof}

Next,  we verify that the rectangular reflection matrix $R$ in \eqref{eq:reflection_def} is weak $\mathcal{S}$.

\begin{lemma}\label{lem: matrix}
    $R=[R_{1},\cdots,R_{k}]$ in \eqref{eq:reflection_def} is weak $\mathcal S$ with respect to $S$ and $f$ for $\eps\in (0,1)$ and $c \in \N \cup \{\infty\}$.
\end{lemma}

\begin{proof}
    Fix an arbitrary subset $I = \{i_1, i_2 , \dots, i_P\} \subseteq \{1,2,\dots,d-1\}$ with $i_1<i_2<\dots<i_P$, and let $J_I = \{j_1,j_2, \dots, j_Q\}$ with $j_1<j_2<\dots<j_Q$. Let
$(N^I R^I)_{p,q} = (NR)_{i_p,j_q}$.  For each $q \in \{1,2,\dots, Q\}$, there exists a unique $p_q \in \{1,2,\dots P\}$ such that $i_{p_q}=\min f(j_q)$. By property (i) in Lemma~\ref{lem: component}, for any $x \in p_{j_q}$, we have $\pi (x)_{i_{p_q}} = 0$, which implies that $(N^IR^I)_{p_q,q}>0$. Moreover, since $\pi(x)_i \ge 1$ for all $i < i_{p_q}$, it follows that $(N^IR^I)_{p,q}>0$ if $i_p = i_{p_q}-1$, whereas if $i_p < i_{p_q}-1$, then $(N^IR^I)_{p_q,q} = 0$. 

 Collecting the above observations, we see that after re-enumerating the column indices $q$ of $N^I R^I$ in increasing order of $i_{p_q}$, the sub-matrix $N^I R^I$ takes the following schematic form: 
   \[
\begin{bmatrix}
+ & + & + & \ge & \ge & 0 & 0 & 0 & 0 \\
\cdot & \cdot & \cdot & + & + & \ge & 0 & 0 & 0 \\
\cdot & \cdot & \cdot & \cdot & \cdot & + & \ge & \ge & \ge \\
\cdot & \cdot & \cdot & \cdot & \cdot & \cdot & + & + & + \\
\cdot & \cdot & \cdot & \cdot & \cdot & \cdot & \cdot &\cdot & \cdot 
\end{bmatrix}
\]
    Here, $"+"$ denotes a positive entry, $"\ge"$ a nonnegative entry, $0$ a zero entry, and $\cdot$ unspecified entry. This schematic representation makes it evident that 
    there exists a nonnegative linear combination of the row vectors of $N^IR^I$ 
    producing a strictly positive vector, i.e, a vector whose elements are positive, which yields the desired result. This verifies that $R$ is weak-$\mathcal{S}$. 
\end{proof}

\subsection{Checking condition (v) in Theorem \ref{thm: main}}

The following proposition is useful for checking condition (v) in Theorem \ref{thm: main}. 
\begin{proposition}\label{prop: martin}
    Let $\check \zeta^n, \check X^n, \eps^n_1, \eps^n_2 \in D([0,\infty), \R^d)$ and $\check Y^n, \eps^n_3 \in D([0,\infty), \R^k)$.  Suppose that $X^n = \check X^n + \eps^n_1$, $Y^n = \check Y^n + \eps^n_2$, $\zeta^n = \check \zeta^n + \eps^n_3$ and $\eps^n_1$, $\eps^n_2$, $\eps^n_3$ converge to $\mathbf 0$ in probability as $n \to \infty$.
    Suppose that $\{\check X^n(t)-\check X^n(0)\}^\infty_{n=1}$ is uniformly integrable for each $t \ge 0$, and for each $n$, $\{\check X^n(t)-\check X^n(0), t\ge 0\}$ is a $P^n$-martingale with respect to the filtration generated by $(\check \zeta^n, \check X^n, \check Y^n)$. Then, if conditions (i)-(iv) hold, condition (v) holds. 
\end{proposition}

\begin{proof}
We can prove this statement by adapting the argument in the proof of \cite[Proposition 4.2]{williams1998invariance}, replacing $C_b(\R^{3d})$ with $C_b(\R^{2d+k})$. More precisely, assume that along a subsequence $Z^{n_k} = (\zeta^{n_k}, X^{n_k}, Y^{n_k}) \Rightarrow Z' = (\zeta', X',Y')$ as $k \to \infty$. Then also $\check{Z}^{n_k} = (\check \zeta^{n_k}, \check X^{n_k}, \check Y^{n_k}) \Rightarrow \check{Z}' = (\zeta', X',Y')$. Let $0 \le s_1 < s_2 < \dots < s_m \le s \le t< \infty$ and $f_1, f_2, \dots, f_m \in C_b(\R^{2d+k})$. By the $\pi$-$\lambda$ theorem, it suffices to show that 
    \begin{align*}
        \E[f_1(Z'(s_1))\cdots f_m(Z'(s_m))\E[( X'(t)-X'(0))|\mathcal F'_s]] = \E[f_1(Z'(s_1))\cdots f_m(Z'(s_m))(X'(s)-X'(0))].
    \end{align*}
    By the assumed weak convergence, uniform integrability of $\{\check X^{n_k}(t)-\check X^{n_k}(0)\}^{\infty}_{k=1}$, and the bounded convergence theorem,
    \begin{align}\label{eq: weak convergence}
    \begin{split}
        \E[f_1(\check Z^{n_k}(s_1)) \cdots f_m(\check Z^{n_k}(s_m))(\check X^{n_k}(t)-\check X^{n_k}(0))] \to &\E[f_1(Z'(s_1)) \cdots f_m(Z'(s_m))(X'(t)-X'(0))] \\ 
        =&\E[f_1(Z'(s_1)) \cdots f_m(Z'(s_m))\E[(X'(t)-X'(0))|\mathcal F'_s].
    \end{split}
    \end{align}
    
    By the assumed martingale property, 
    \begin{align*}
        \E[f_1(\check Z^{n_k}(s_1)) \cdots f_m(\check Z^{n_k}(s_m))(\check X^{n_k}(t)-\check X^{n_k}(0))] 
        = &\E[f_1(\check Z^{n_k}(s_1)) \cdots f_m(\check Z^{n_k}(s_m))\E[\check X^{n_k}(t)-\check X^{n_k}(0)|\check{\mathcal F}_s]]\\
        = &\E[f_1(\check Z^{n_k}(s_1)) \cdots f_m(\check Z^{n_k}(s_m))(\check X^{n_k}(s)-\check X^{n_k}(0))].
    \end{align*}
    Taking $s = t$ in \eqref{eq: weak convergence} gives
    \begin{align*}
        \E[f_1(\check Z^{n_k}(s_1)) \cdots f_m(\check Z^{n_k}(s_m))(\check X^{n_k}(s)-\check X^{n_k}(0))]\to \E[f_1(Z(s_1)) \cdots f_m(Z'(s_m))( X'(s)-X(0))].
    \end{align*}
    Comparing this with \eqref{eq: weak convergence}, we see that the last terms coincide, and hence the desired identity holds.
\end{proof}

Now it is enough to verify the hypothesis of Proposition~\ref{prop: martin} to verify (v) in Theorem \ref{thm: main}. 
More precisely, by Proposition~\ref{prop: error}, we can set $\check \zeta^n = \zeta^n$, $\check Y^n = Y^n$, and $\check X^n(t)  = n^{-1/2}(X(\lfloor nt\rfloor)-nt)$.
For fixed $t \ge 0$, $\{\check X^n(t)-\check X^n(0)\}^\infty_{n=1} = \{n^{-1/2} (X(\lfloor nt\rfloor)-nt)\}^\infty_{n=1}$ is bounded in $L_2$, hence uniformly integrable.  
By the construction of the SBBS and the bulk process, for each $n$, $(\check \zeta^n, \check X^n, \check Y^n)$ is adapted to the filtration $\mathcal F_t$ generated by $\mathcal{F}_t= \sigma\{\eta_0,...,\eta_{\lfloor t-1\rfloor}\}$ with $\mathcal F_0 = \sigma(\{\emptyset\})$, and $\check X^n$ is a $\mathcal F_t$-martingale. Thus, $\check X^n$ is a martingale with respect to the filtration generated by $(\check \zeta^n, \check X^n, \check Y^n)$.

\subsection{Checking Condition (iii) in Theorem \ref{thm: main}}

It remains to verify Condition (iii) in Theorem \ref{thm: main}. Since $(X_t-nt)_{t \in \N_0}$ has i.i.d. increments with mean $\mathbf 0$ and covariance matrix $\eps(1-\eps)I_d$, 
we can apply Donsker's theorem to show that $(n^{-1/2}\overline X_{nt}:t \ge 0)$ converges weakly to a Brownian motion with mean $\mathbf 0$, covariance matrix $\eps(1-\eps) I_d$, and initial distribution $\delta_{\boldsymbol{0}}$. 
Hence, condition (iii) follows from the following proposition: 
\begin{proposition}\label{prop: error}
    As $n \to \infty$, 
    \begin{align*}
        \left(\frac{\overline \alpha_{nt}}{\sqrt n}:t \ge 0\right) \rightarrow \mathbf{0} \quad \text{in probability.}
    \end{align*}    
\end{proposition}

\begin{proof}
Note that for each $t \in \Z_{\ge 0}$,
    \begin{align*}
        \alpha_t = \sum_{s = 0}^{t-1} \sum_{j=1}^k \mathbf{1}(\zeta_s \in p_j)(\Delta A_s-R_j).
    \end{align*}
    By definition, $(\alpha_t)_{t \in \Z_{\ge 0}}$ is a discrete-time martingale with respect to a filtration $\{\mathcal G_t\}_{t \in \Z_{\ge 0}}$, where $\mathcal G_t = \sigma(\{\eta_s: 0 \le s\le t-1\})$.
    Moreover, $( \alpha_{\lfloor t\rfloor})_{t \ge 0}$ is a continuous-time martingale with respect to a filtration $\{\bar {\mathcal G}_t\}_{t \ge 0}$, where $\bar {\mathcal G}_t = \mathcal G_{\lfloor t \rfloor}$.
    Because, for $T \geq 0,$ $\sup_{0 \leq t \leq T }|n^{-1/2}\alpha_{\lfloor nt\rfloor}-n^{-1/2}\overline{\alpha}_{nt}|\leq n^{-1/2}c$ for some constant $c$ that is independent of $T,$ it suffices to show that, the process $n^{-1/2}\alpha_{\lfloor nt\rfloor}$ converges to $0$ in probability uniformly on $[0,T]$ for each $T>0.$
We can apply Doob's inequality to obtain that, for each $T>0$, 
\begin{align*}
    \E\left[\sup_{t \in [0,T]}\lVert \frac{ \alpha_{\lfloor nt \rfloor }}{\sqrt{n}} \rVert^2 \right] \le \frac{4}{n}\E[\lVert\overline \alpha_{\lfloor nT\rfloor }\rVert^2]
\end{align*}
Note that $\E[\lVert \alpha_{\lfloor nT\rfloor }\rVert^2] = \E[\langle \alpha \rangle_{\lfloor nT\rfloor }]$, where $\langle \alpha \rangle_{\lfloor \cdot \rfloor}$ is the predictable quadratic covariation of $\alpha_{\lfloor \cdot \rfloor }$ and 
\begin{align*}
    \langle  \alpha \rangle_{\lfloor t\rfloor } &\le \sum_{s = 0}^{\lceil t \rceil -1}  \sum_{j=1}^k \mathbf{1}(\zeta_s \in p_j)\lVert\Delta A_s-R_j\rVert^2 = \sum_{j=1}^k\sum_{i=1}^{N^j_{\lceil t \rceil -1}}\lVert\Delta A_{\tau^j_i}-R_j\rVert^2,
\end{align*}
where $\tau^j_i$ is the $i$th hitting time of $\zeta$ on $p_j$ and $N^j_s$ is the number that $\zeta$ hits $p_j$ until time $s$. Since $\Delta A_t$ is bounded uniformly on $t \in \Z_{\ge 0}$, there exists a constant $C>0$ such that $\langle \alpha\rangle_{\lfloor t\rfloor } \le CN_{\lceil t \rceil -1}$, where $N_s:=\sum_{j=1}^kN^j_s$ Hence, it follows from Theorem \ref{thm:upper_bd_local_time_bdry} that 
\begin{align*}
    \mathbf 0 \le \frac{\E[\langle \alpha \rangle_{\lfloor nT\rfloor }]}{n} \le C\frac{\E[N_{\lceil nT \rceil - 1}]}{n} \to \mathbf 0 \quad \text{as} \ n \to \infty.
\end{align*}
Consequently, we obtain 
\begin{align*}
    \E\left[\sup_{t \in [0,T]}\lVert \frac{ \alpha_{\lfloor nt\rfloor }}{\sqrt{n}} \rVert^2 \right] \to \mathbf 0 \quad \text{as} \ n \to \infty
\end{align*}
for each $T > 0$, which completes this proof.
\end{proof}

\subsection{Proof of Lemma \ref{lem: component}}

We conclude this section by providing a proof of Lemma \ref{lem: component}. 

\begin{proof}[\textbf{Proof of Lemma \ref{lem: component}}]
    First of all, we assume $c \ge d$. For $y_1,y_2 \in \partial \mathbb{S}^{d}$, we define $y_1 \sim y_2$ to mean that $\zeta$ admits the same transition kernel at $y_1$ and $y_2$. This clearly defines an equivalence relation on $\partial \mathbb{S}^{d}$, and we denote the equivalence class of a point $x \in \partial \mathbb{S}^{d}$ by $[y]$. 
    The desired partition of $\partial \mathbb{S}^{d}$ is precisely the collection of these equivalence classes. Indeed, 
    for each equivalence class $[y]$, $\pi([y])$ takes a form of a certain "box" in $\partial \Z^{d-1}_{\ge 1}$(the boundary of $\Z^{d-1}_{\ge 0}$), which can be described more explicitly by transforming $x=\pi(y)$ through the following iterative procedure: 
    \begin{enumerate}
        \item \textbf{Initialization:} Set $i_0 = i'_0 = 0$. 
        \item \textbf{Iteration:} 
        \begin{itemize}
        \item [(a)] Define 
        \begin{align*}
            i_m:=\min\{i>i_{m-1}+i'_{m-1}:x_i = 0\}, \quad \text{with} \ \min \emptyset := \infty.
        \end{align*}
        \item [(b)] Update the coordinates of $x$ as follows: 
        \begin{itemize}
            \item If $i_{m-1}+i'_{m-1}+1= i<i_m$, set $x_i := i'_{m-1}+1$.
            \item If $i_{m-1}+i'_{m-1}+2\le i < i_m $, set $x_i := 1$. 
        \end{itemize} 
        \item [(c)] If $i_m = \infty$, terminate the procedure. 
        \item [(d)] Otherwise, let $i'_{m}$ be the largest nonnegative integer such that 
        $$ x_{i_m} = 0, \ x_{i_m}+x_{i_m+1} \le 1,\  x_{i_m+1}+x_{i_m+2} \le 2, \  \dots, x_{i_m+1}+\cdots+x_{i_m+i'_m}\le i'_m.$$
        Repeat for $m+1$ from step $(i)$. 
    \end{itemize}
    \end{enumerate}
    
    Let $\bar x$ denote the resulting vector after this process. Define the index set 
    \begin{align}\label{eq: I}
        I = \{1 \le i \le d-1: i_m \le i \le i_m+i'_m \ \text{for some $m$}\},
    \end{align}
    which consists of coordinates that were not modified during the procedure. 

    The equivalence class of $y$ can then be characterized as 
    \begin{align}\label{eq: class}
        [y] = \{y' \in \partial\mathbb{S}^{d}: \pi(y')_i = \bar x_i \ \text{for} \ i\in I, \ \text{and} \ \pi(y')_i \ge \bar x_i \ \text{for} \ i \in I^c\}.
    \end{align}
    In other words, the coordinates of $\pi(y')$ indexed by $I$ must match those of 
$\bar x$, while the coordinates in $I^c$
 may take any value greater than or equal to the corresponding entry of $\bar x$. 

    To verify that this characterization is exact, we now prove that the above two sets indeed coincide. First of all, we show that the transition kernel of $\zeta$ is homogeneous on the right-hand side of \eqref{eq: class}. Consider a configuration of balls corresponding to an element of this set. By the construction of the set, any $i$th leftmost ball whose index $i$ does not belong to $\{1 \le i\le d: i_m \le i \le i_m + i'_m+1 \text{ for some $m$}\}$ is sufficiently separate from the others and therefore does not interact with them under any coin flips. Moreover, for each $m$, the gaps between the balls with indices $i_m \le i \le i_m + i'_m+1$ remain homogeneous across all elements of the set. Consequently, the transition kernel is homogeneous on the entire set, which proves $\supseteq$ in \eqref{eq: class}. Indeed, this argument doesn't depend on $c$, which means the homogeneity within each $p_j$ holds for any $c$.   

Conversely, let $\hat x:=\pi(\hat y)$ with $\hat y \in \partial \mathbb{S}^{d}$ satisfy at least one of the following conditions: 
 \begin{itemize}
        \item $\hat x_i \not = \bar x_i$ for some $i \in I$, or 
        \item $\hat x_i< \bar x_i$ for some $i \in I^c$. 
    \end{itemize}  
    Let $\hat i$ be the smallest index where this discrepancy occurs. 
\begin{itemize}
    \item If $\hat i \in I$, then the $\hat i+1$st leftmost ball can interact with the preceding balls in the configuration corresponding to $\bar x$, i.e, the gaps between balls are represented by $\bar x$, and this interaction changes when $\bar x_{\hat i}$ is replaced by $\hat x_{\hat i}$.
    \item If $\hat i \in I^c$, then the $\hat i+1$st leftmost ball does not interact with the preceding balls in the configuration corresponding to $\bar x$, whereas it can interact once $\bar x_{\hat i}$ is replaced by $\hat x_{\hat i}$.  
\end{itemize}
Hence, property (iii) follows from this argument, which depends on the assumption $c \ge d$.   

    It follows that the number of equivalence classes, that is, the number of cells $k$, equals the number of possible vectors $\bar x$. It can be easily verified that these vectors satisfy 
    \begin{align*}
        \bar x_1 = 0, \bar x_1+\bar x_2\le 1, \dots, \bar x_1 + \bar x_2 + \cdots+\bar x_{d-1} \le d-2.
    \end{align*}
    Hence, $k \le (d-1)!$, which is finite.
    
    Each boundary cell $p_j$ corresponds to an equivalence class $[y]$ for some $y \in \partial\mathbb{S}^{d}$, and thus has an associated index set $I$ as defined in \eqref{eq: I}. By the definition of the function $f$, $f(j) = I$. Note that $I$ is always nonempty. Indeed, if $I = \emptyset$, then all coordinates of $y$ would satisfy $\pi(y)_i \ge 1$ for each $1 \le i \le d-1$, which contradicts to the assumption that $y \in \partial \mathbb{S}^{d}$ (i.e., $y$ lies on the boundary). Therefore, $f(j)$ is nonempty for all $j$. 
    
    Furthermore, for each $i \in \{1,2,\dots, d-1\}$, $f(j) = \{i\}$ if and only if the corresponding vector $\bar x$ of $p_j$ is $(1,1,\dots, 0,2,1,1,\dots,1)$, where $0$ appears in the $i$th coordinate. 
    By the previous rearrangement of indices, such a class corresponds to $p_i$. 
\end{proof}

\section{Proofs of Theorem~\ref{thm: pushtasepSRBM}}

In this section, we prove Theorem 
~\ref{thm: pushtasepSRBM}, which gives an SRBM limit of the gap process associated with PushTASEP. The overall strategy is the same as the proof of Theorem \ref{thm: SRBM_weak_convergence}, which proceeded by invoking Theorem~\ref{thm: main}.

\begin{proof}[\textbf{Proof of Theorem~\ref{thm: pushtasepSRBM}}]
    Let $\hat\xi$ be the embedded Markov chain of $\xi$, and $J$ be the associated jump process, that is, $\xi_t = \hat\xi_{J(t)}$ for $t \ge 0$. Here, $J(t)$ is a Poisson process with rate $d$.  
    Similarly as in Lemma~\ref{lem: component}, the boundary $\partial \mathbb S^d$ is partitioned into $2^{d-1}-1$ boundary cells $p_j$, where the transition kernel of $\xi$ is homogeneous, as follows: 
    \begin{itemize}
        \item [(i)] Each $j \in \{1,2,\dots, 2^{d-1}-1\}$ one-to-one corresponds to a nonempty subset $I^j \in \mathcal P(\{1,2,\dots,d-1\})$ and 
        \begin{align*}
            p_j = \{x \in \partial \mathbb S^d: \pi(x)_i = 0 \text{ if and only if }i \in I^j\},
        \end{align*}
        Where the map $\pi:\mathbb{S}^{d}\rightarrow \Z_{\ge 0}^{d-1}$ taking a ball configuration to the associated gap configuration is defined at \eqref{eq:def_gap_process}. 
        \item [(ii)] Moreover, if $1\le j \le d-1$, $I^j = \{j\}$.
    \end{itemize}
    When $\hat{\xi}_{n}\in \mathbb{S}^{d}\setminus \partial \mathbb{S}^{d}$, then it behaves as the following ``bulk process'' $\hat{X}_{n}$: $\hat{X}_{n+1}-\hat{X}_{n}=e_{I_{n}}$, where $I_{n}$ is the uniformly random index of the ball that makes the $n$th jump. Hence $\hat{X}_{n}$ is a random walk on $\Z^{d}$ with increments distributed as  $\textup{Uniform}(e_{1},\dots,e_{d})$. We keep track of the discrepancy between $\hat{\xi}$ and $\hat{X}$ as 
     \begin{align*}
        \Delta \hat \xi_n-\Delta \hat X_n = \sum_{j=1}^{2^{d-1}-1}\mathbf{1}(\hat \xi_n \in p_j)(\Delta \hat \xi_n-\Delta \hat X_n).
    \end{align*}
    By the homogeneity of the transition kernel on each boundary cell, the mean reflection vector $R_j:= \E[\Delta \hat \xi_n-\Delta \hat X_n|\hat \xi_n \in p_j]$ is well-defined. 
    Decomposing the boundary correction into the expected amount and fluctuation, we can further write 
    \begin{align*}
        \Delta \hat \xi_n = \Delta \hat X_n+R\Delta \hat Y_n+\underbrace{\sum_{j=1}^{2^{d-1}-1}\mathbf{1}(\hat \xi_n \in p_j)(\Delta \hat \xi_n - \Delta \hat X_n-R_j\Delta \hat Y^{j}_n)}_{=:\delta_n},
    \end{align*}
    where $R = [R_1,R_2,\dots, R_{2^{d-1}-1}]$ and $\hat Y = (\hat Y^1, \hat Y^2, \dots, \hat Y^{2^{d-1}-1})^{\top}$ with $\hat Y^j_n = \sum_{s=0}^{n-1} \mathbf(\hat \xi_s \in p_j)$. By integrating both sides and replacing $n$ by $J(t)$, we obtain the following overdetermined Skorokhod decomposition for $\xi$: 
    \begin{align*}
        \xi_t = X_t + RY_t + \alpha_t,
    \end{align*}
    where $X_t := \hat X_{J(t)}$, $Y_t := \hat Y_{J(t)}$, and $\alpha_t := \sum_{s=0}^{J(t)-1}\delta_s$. Note that $(X_t)_{t \ge 0}$ is a $d$-dimensional Poisson process whose entries are independent unit Poisson processes with rate 1. 
    
    We need one more step before invoking Theorem~\ref{thm: main}. Let $\overline\xi, \overline X, \overline \alpha, \overline Y$ denote the piecewise linear interpolation of the jump processes of   $\xi, X, \alpha, Y$, respectively.  Clearly, the above Skorokhod decomposition is preserved under linear interpolation:
    \begin{align*}
        \overline{\xi}_{t} = \overline{X}_t + \overline{\alpha}_t  + R\overline{Y}_t. 
    \end{align*}   
    Now let $\xi^n_t:=n^{-1/2}(\overline{\xi}_{nt}-nt\mathbb{1}_d)$, $X^n_t:=n^{-1/2}(\overline X_{nt}-nt\mathbb{1}_d + \overline{\alpha}_{t})$, and $Y^n_t: = n^{-1/2}(\overline Y_{nt})$. Note that the difference between the c\`{a}dl\`{a}g process $\xi$ and the corresponding piecewise linear process $\overline\xi$ is uniformly bounded for all times, so $(n^{-1/2}(\overline \xi_{nt} - \xi_{nt}):t\ge 0) \to \mathbf{0}$ in probability as $n\rightarrow\infty$. Hence, to obtain the desired result, it suffices to apply Theorem~\ref{thm: main} to $(\xi^n, X^n, Y^n)$.

    To this effect, we set $f: \{1,2,\dots, 2^{d-1}-1\}\to \mathcal P(\{1,2,\dots, d-1\})$ by $f(j) = I^j$.  Then, it can be easily shown that $R$ is weak $\mathcal S$ with respect to $f$, and conditions (1) and (2) are satisfied. Also,  for $1\le j\le d-1$, since $\hat{\xi}_{n}\in p_{j}$ means the $j$th and the $j+1$st balls are next to each other and all other balls are separated by distance $\ge 2$, the first $d-1$ expected reflection vectors are given by  
    \begin{align*}
        R_j = 
        \frac{1}{d} (e_{j}+e_{j+1}) - \frac{1}{d} e_{j} = \frac{1}{d} e_{j+1}. 
    \end{align*}
    Hence $\hat R = [R_1, R_2,\dots, R_{d-1}] = R_{PT}$, where $R_{PT}$ is defined in \eqref{eq:sigma_R}. 
    Moreover, conditions (i), (ii), and (iv) follow in the same manner as for the SBBS case shown in the proof of Theorem \ref{thm: SRBM_weak_convergence}. 
    
    By 
    a functional central limit theorem, the centered Poisson process $X_{nt}-nt\mathbb{1}_d$ converges in diffusive scaling to the $d$-dimensional Brownian motion with zero drift, 
    covariance matrix $I_d$, and initial distribution $\delta_{\mathbf 0}$. 
    We can also show that $(n^{-1/2}\overline \alpha_{nt}:t\ge 0)$ converges to $\mathbf 0$ in probability as $n\rightarrow \infty$ by using the same argument as in the proof of Proposition~\ref{prop: error}, where Theorem~\ref{thm:upper_bd_local_time_bdry} is replaced by Theorem~\ref{thm:upper_bd_local_time_bdry_PushTASEP}. Since the difference between $X_{t}$ and $\overline{X}_{t}$ stays uniformly bounded, it follows that condition (iii) in Theorem \ref{thm: main} follows.

    Finally, it remains to check condition (v) in Theorem \ref{thm: main}. Here, we apply Proposition~\ref{prop: martin} by setting $\check \xi_t^n := n^{-1/2}(\xi_{nt}-nt\mathbb{1}_d)$, $\check X_t^n := n^{-1/2}(X_{nt}-nt\mathbb{1}_d)$, and $\check Y_t^n := n^{-1/2}Y_{nt}$. $\xi^n-\check \xi^n$, $X^n-\check X^n$, and $Y^n-\check Y^n$ converges to $\mathbf 0$ almost surely as $n \to \infty$ since $\sup_{s\ge 0} \| X_{s}-\overline{X}_{s} \|_{1}\le d$ and $\sup_{s\ge 0} \| Y_{s}-\overline{Y}_{s} \|_{1}\le d$. For each $t \ge 0$, $\{\check X^n_t - \check X^n_0\}_{n=1}^\infty$ is $L^2$- bounded, which implies uniform integrability. Moreover, note that for each $n$, the increment of $\check X^n(t + \Delta t) - \check X^n(t)$ is independent of $(\check \xi^n(s),\check X^n(s),\check Y^n(s)):0 \le s \le t)$, which implies that $\check X$, the centered Poisson process, is a martingale with filtration generated by $(\check \xi^n,\check X^n,\check Y^n)$.
    In conclusion, it follows that $\xi^n$ weakly converges to the SRBM associated with $(S,\mathbf 0, I_d, R_{PT}, \delta_{\mathbf 0})$.
\end{proof}

\section{Proof of Theorem~\ref{thm: main} 
} \label{sec:proof_SRBM_invariance}

In this section, we will prove the extended SRBM invariance principle stated in Theorem \ref{thm: main}. We will proceed by following the scheme below. 

\begin{description}
    \item[Step 1.]  We show that $(W^n,X^n, Y^n)_{n \in \N}$ is tight, which implies (relative) compactness of $\{(W^n,X^n, Y^n):n \in \N\} \subset \mathcal P(C([0,\infty), \R^{2d+k}))$, where $\mathcal P(C([0,\infty), \R^{2d+k}))$ is the set of probability measures on $C([0,\infty), \R^{2d+k})$. 
\end{description}

After \textbf{Step 1}, it remains to show that all the limit points have the same distribution, which is the SRBM $\hat W$ in Theorem \ref{thm: main}.

\begin{description}
    \item[Step 2.] We show that for each limit point $(W',X',Y')$ of $(W^n,X^n, Y^n)_{n \in \N}$, the following properties hold: 
\begin{itemize}
    \item [(i)] $W' = X' + R Y'$ a.s.,
    \item [(ii)] $W'$ has continuous paths and $W'(t) \in S$ for all $t \ge 0$ a.s.,
    \item [(iii)] 
    \begin{itemize}
        \item [(a)] $X'$ is a $d$-dimensional Brownian motion with drift $\mathbf 0$ and covariance matrix $\Sigma$, and initial distribution $\delta_{\mathbf 0}$,
        \item [(b)] $\{X'(t)-X'(0),\mathcal F_t,t\ge0\}$ is a martingale, where $\mathcal F_t = \sigma\{(W'(s),X'(s),Y'(s)):0 \le s \le t\}$, 
    \end{itemize}
    \item [(iv)] $Y'$ is an $\{\mathcal F_t\}$-adapted $k$-dimensional process such that a.s. for $j \in \{1,2,\dots, k\}$,
    \begin{itemize}
        \item [(a)] $Y'_j(0)=0$,
        \item [(b)] $Y'_j$ is continuous and non-decreasing, 
        \item [(c)] $\int_0^{t}\mathbf{1}_{F_{f(j)}}(W'(s))dY'_j(s)=Y'_j(t)$ for all $t \ge 0$.
    \end{itemize}
\end{itemize}

\item[Step 3.] We show that a.s. $Y'_j \equiv 0$ if $j > m$ so that $W' = X' + \hat R \hat Y'$, where $\hat Y' = (Y'_1,Y'_2,\dots, Y'_m)^\top$. For this, we use a minor modification of the proof of \cite[Lemma 4, Lemma 5]{reiman1988boundary}.  

\item[Step 4.]  From the assumption on $R$, $\hat R$ automatically satisfies $\mathcal S$-condition. Hence, by the uniqueness of SRBM associated with a $\mathcal S$-reflection matrix, all the limit points of $(W^n, X^n, Y^n)_{n \in \N_0}$ are the same in distribution, which concludes our result. 
\end{description}

\textbf{Steps 1, 2}, and \textbf{4} follow the same procedure as in the original invariance principle of \cite{williams1998invariance,kang2007invariance}, except that we must additionally handle boundary components of codimension at least two, which give rise to extra reflection vectors and pushing processes. In this overdetermined Skorokhod decomposition setting, we introduce \textbf{Step 3} as the key new ingredient. Moreover, the original oscillation inequality of \cite{williams1998invariance, kang2007invariance} cannot be applied directly to prove \textbf{Step 1}, so we adapt this inequality to our context.

The key point in the proofs of \textbf{Step 1} and \textbf{Step 3} is that the weak $\mathcal S$-condition is sufficient in place of the $\mathcal S$-condition.

\subsection{Step 1}
We first establish tightness of the sequence $(W^n, X^n, Y^n)_{n \in \N}$ using the following oscillation inequalities: 
\begin{proposition}[Extended Oscillation Inequality]\label{prop: osc}
    Let $R$ be a $d \times k$ weak $\mathcal S$ matrix with respect to $S = \{x \in \R^d: n_i \cdot x \ge b_i \ \text{for all }1\le i \le m\}$ and a function $f:\{1,2,\dots, k\} \to \mathcal P(\{1,2,\dots, d\})$ such that \begin{itemize}
        \item [(1)] $f(j) \not = \emptyset$ for all $j \in \{1,2,\dots, k\}$,
        \item [(2)] for each $i \in \{1,2,\dots, d\}$, $f(j) = \{i\}$ if and only if $j = i$.
    \end{itemize}

    Suppose that $\delta > 0$, $0 \le t_1 < t_2<\infty$ and $w,x\in C([t_1,t_2], \R^d)$, $y \in C([t_1,t_2], \R^k)$ are such that 
    \begin{itemize}
        \item [(i)] $w(t) = x(t) + Ry(t)$ for all $t \in [t_1, t_2]$, 
        \item [(ii)] $w(t) \in S$ for all $t \in [t_1, t_2]$,
        \item [(iii)] for $j \in \{1,2,\dots, k\}$, 
        \begin{itemize}
            \item [(a)] $y_j(t_1) \ge 0$,
            \item [(b)] $y_j$ is non-decreasing, 
            \item [(c)] $\int^{t_2}_{t_1} \mathbf{1}_{F^{\delta}_{f(j)}}(w(s)) dy_j(s) = y_j(t_2)-y_j(t_1)$.
        \end{itemize}
    \end{itemize}
    Then, there exists a constant $C >0$ depending only on $R$, $S$, and $f$ such that 
    \begin{align}
        Osc(y,[t_1,t_2]) &\le C(Osc(x, [t_1, t_2])+\delta), \label{eq: osc1}\\
        Osc(w,[t_1,t_2]) &\le C(Osc(x, [t_1, t_2])+\delta),\label{eq: osc2}
    \end{align}
    where \begin{align*}
        Osc(g, [t_1,t_2]):=\sup\{\max_{1 \le \ell \le n}|g_\ell(t)-g_\ell(s)|: t_1 \le s < t \le t_2\}
    \end{align*}
    for any $g \in C([t_1, t_2], \R^n)$ and positive integer $n$. 
\end{proposition}

\begin{proof}

    This proof closely follows the proof of \cite[Lemma 4.3]{dai1996existence} and \cite[Theorem 5.1]{williams1998invariance}, which proceeds by induction on the number of faces $m$ of $S$. If $m = 1$, then by assumptions (1) and (2), $k=1$, i.e, $R$ consists of only one column $v_1$. Hence, for any $w,x,y$ satisfying (i)--(iii), 
    \begin{align*}
        n_1^\top w(t) = n_1^\top x(t) + n_1^\top Ry(t) = n_1^\top x(t) + n_1^\top v_1y(t), 
    \end{align*}
    where $n_1^\top v_1>0$ by the weak $\mathcal S$-condition. Since $n_1^\top w(t)\ge b_1$ for all $t \in [t_1,t_2]$ and $n_1^\top v_1y(\cdot)$ is a non-decreasing function that increases only when $b_1 + \delta \ge n_1^\top w(\cdot) \ge b_1$. Hence, as in the proof of \cite[Theorem 5.1]{williams1998invariance}, we can obtain 
    \begin{align*}
        Osc(y,[t_1,t_2]) \le \frac{1}{n_1^\top v_1} (Osc(n_1^\top x,[t_1,t_2])+\delta) \le \frac{\max(\lVert n_1 \rVert_1,1)}{n_1^\top v_1}(Osc(x, [t_1,t_2])+\delta), 
    \end{align*}
    where $\lVert \cdot \rVert_1$ denotes the $L^1$-norm of a vector. Since $w= x + Ry = x+ v_1 y$, we also have 
    \begin{align*}
        Osc(w,[t_1,t_2]) \le Osc(x,[t_1,t_2]) + \lVert v_1 \rVert_1 Osc(y,[t_1,t_2]) \le (1 + \lVert v_1 \rVert_1 \frac{\max(\lVert n_1 \rVert_1,1)}{n_1^\top v_1})(Osc(x, [t_1,t_2])+\delta).  
    \end{align*}
    Thus, we can take a positive constant so that  \eqref{eq: osc1} and \eqref{eq: osc2} hold for $m=1$. 

Now, suppose that it holds for $m < \ell$ for some fixed $\ell \ge 2$. Let $m = \ell$ and $w,x,y$ satisfy (i)--(iii). First of all, we further assume that there exists an empty or proper maximal subset $I$ of $\{1,2,\dots, \ell\}$ such that $y_j$ does not increase on $[t_1,t_2]$ for $f(j) \not \subseteq I$. If $I = \emptyset$, then \eqref{eq: osc1} and \eqref{eq: osc2} trivially holds with $C=1$. If $I \neq \emptyset$, then $w,x,y$ satisfy (i)--(iii) with $R^I=[R_j]_{f(j) \subseteq I}$, $S^I:=\{x \in S: n_i \cdot x \ge b_i \text{ for all $i \in I$}\}$, $f^I:=f|_{\{j \in \{1,2,\dots, k\}:f(j) \subseteq I \}}$, which together satisfy weak $\mathcal S$-condition and assumptions (1) and (2). Then by assumption, \eqref{eq: osc1} and \eqref{eq: osc2} holds for some constant $C_I$ that only depends on $R^I$, $S^I$, and $f^I$. By taking the maximum of $C_I$'s across all the maximal subsets $I$, we have $C_1$ with which \eqref{eq: osc1} and \eqref{eq: osc2} hold independently of $I$.

Now, we consider $w,x,y$ satisfying (1)--(3) without further restriction.  
Let $K>0$ be a constant to be determined and denote $\eta =Osc(x,[t_1,t_2])+\delta$. By \cite[Lemma 4.1]{dai1996existence}, for each $w$, there exists an empty or maximal subset $I$ of $\{1,2,\dots, \ell\}$ such that 
\begin{align*}
  0 \le n_i \cdot w(t_1) -b_i \le K'\eta \quad \text{for all $i \in I$} \quad \text{and} \quad n_i \cdot w(t_1) -b_i > K\eta \quad \text{for all $i \in I^c$},
\end{align*}
where $K' :=C'\ell K$ and $C'>0$ is a constant that depends only on $S$.
We will then proceed by considering two cases. 

\begin{description}[itemsep=0.1cm]
    \item \textit{Case 1: $I \neq \{1,2,\dots, \ell\}$.}    
    
    For each $i \not \in I$, let $\tau_i := \inf\{t \in [t_1,t_2]: n_i \cdot w(t)-b_i \le \delta\}$ and assume that $\tau_i \le \infty$, i.e., $\tau_i \in [t_1,t_2]$. Then by continuity of $w$, 
    \begin{align*}
        \Osc(w, [t_1, \tau]) = \lim_{t'_2 \uparrow \tau} \Osc(w, [t_1, t'_2]) \le C_{1}\eta.
    \end{align*}
    If we take $K \ge C_1 + 2$, then we obtain
    \begin{align*}
        C_1 \eta \ge Osc(w,[t_1,\tau]) \ge n_i(w(t_1)-w(\tau_i)) \ge K \eta - \delta,
    \end{align*}
    which deduces $\delta \ge 2\delta$, which is a contradiction. Hence, $\tau_i = \infty$ for all $i \not\in I$, and \eqref{eq: osc1} and \eqref{eq: osc2} hold with $C_1$.

    \item{} \textit{Case 2: $I=\{1,2,\dots, \ell\}$.}   

    We split this case into two sub-cases:
\begin{description}
    \item[] \textit{Case 2-(a): $0 \le n_i \cdot w(t)-b_i \le K'\eta$ for all $t \in [t_1,t_2]$ and $i \in I$.}
    
        Since $R$ is weak $\mathcal S$, there exists a vector $\lambda \in \R^{\ell}_{> 0}$ such that $\lambda^\top NR>\mathbf 0$. Then 
        \begin{align*}
            \lambda^\top Nw(t) = \lambda^\top Nx(t) + \lambda^\top NRy(t) \quad \text{for all $t \in [t_1,t_2].$}
        \end{align*}
    Since $y$ is non-decreasing, 
    \begin{align*}
        \min_{i \in \{1,2,\dots,\ell\}}(\lambda^\top NR)_i Osc(y,[t_1,t_2]) \le \min_{i \in \{1,2,\dots,\ell\}}(\lambda^\top NR)_i Osc(\lVert y \rVert_1,[t_1,t_2]) \\ \le Osc(\lambda^\top Nw,[t_1,t_2])+Osc(\lambda^\top Nx, [t_1,t_2]) \\ \le \sum_{i=1}^\ell \lambda_i(Osc(Nw,[t_1,t_2])+Osc(Nx, [t_1,t_2])) \le \sum_{i=1}^\ell \lambda_i(K'+\max_{i \in \{1,2,\dots, \ell\}}\lVert n_i \rVert_1)\eta,
    \end{align*}
    which implies that \eqref{eq: osc1} holds. From the fact that $y = x + Ry$, \eqref{eq: osc2} also follows. 
    \item[] \textit{Case 2-(b): There is $t \in [t_1, t_2]$ and $i \in \{1,2 ,\dots, \ell\}$ such that $n_i \cdot w(t)-b_i > K'\eta$.}

    Let $\tau = \inf\{t \in [t_1, t_2]: n_i \cdot w(t)-b_i > K'\eta \ \text{for some } i\}$, which is finite by assumption.
\end{description}

  As we argued in case 2-(a), we can deduce from the continuity of $w, y$ that 
\begin{equation} \label{eq: po13}
\Osc(w, [t_1, \tau]), \ \Osc(y, [t_1, \tau]) \le C_2\eta
\end{equation}
for some $C_2>0$. 
Since $\tau < \infty$, then $n_i \cdot w(\tau)-b_i \ge K'\eta$ for some $i \in \{1,2,\dots, \ell\}$. By applying \cite[Lemma 4.1]{dai1996existence} again on $w(\tau)$, we can use the analysis of case 1, which implies that 
\begin{align}\label{eq: po14}
     \Osc(w, [\tau, t_2]), \ \Osc(y, [\tau, t_2]) \le C_1\eta.
\end{align}
Combining \eqref{eq: po13} and \eqref{eq: po14}, we obtain \eqref{eq: osc1} and \eqref{eq: osc2} by setting $C =C_1 + C_2$, which depends only on $R$, $S$, and $f$. 
\end{description}

\end{proof}
\begin{lemma}
    $(W^n,X^n, Y^n)_{n \in \N}$ is tight in $\mathcal P(C([0,\infty], \R^{2d+k})$ that is the space of probability measures on $C([0,\infty), \R^{2d+k})$. 
\end{lemma}
\begin{proof}
    Since $(X^n)_{n \in \N}$ is weakly convergent, it is tight. By Proposition \ref{prop: osc}, tightness of $(X^n)_{n \in \N}$ gives tightness of $(W^n)_{n \in \N}$ hence that of $(W^n, X^n, Y^n)_{n \in \N}$.
\end{proof}

By Prohorov's theorem, $\{(W^n,X^n, Y^n):n \in \N\}$ is relatively compact in $\mathcal P(C([0,\infty], \R^{2d+k})$, which is a metric space equipped with Prohorov metric. 
\subsection{Step 2}
Next, we show Step 2. These properties imply that $W'$ is an SRBM associated with $(S, \mathbf{0}, \Sigma, \hat R, \delta_{\mathbf 0})$, assuming that $Y'_j \equiv 0$ for each $j > m$, which will be proved in Step 3.
\begin{lemma}\label{lem: main1}
    Every limit point $(W',X',Y')$ of $(W^n, X^n, Y^n)_{n \in \N}$ satisfies the following properties: 
    \begin{enumerate}
    \item [(i)] $W' = X' + R Y'$ a.s.,
    \item [(ii)] $W'$ has continuous paths and $W'(t) \in S$ for all $t \ge 0$ a.s.,
    \item [(iii)] 
    \begin{itemize}
        \item [(a)] $X'$ is a $d$-dimensional Brownian motion with drift $\mathbf 0$ and covariance matrix $\Sigma$, and initial distribution $\delta_\mathbf{0}$,
        \item [(b)] $\{X'(t)-X'(0),\mathcal F_t,t\ge0\}$ is a martingale, where $\mathcal F_t = \sigma\{(W'(s),X'(s),Y'(s)):0 \le s \le t\}$, 
    \end{itemize}
    \item [(iv)] $Y'$ is an $\{\mathcal F_t\}$-adapted $k$-dimensional process such that a.s. for $j \in \{1,2,\dots, k\}$,
    \begin{itemize}
        \item [(a)] $Y'_j(0)=0$,
        \item [(b)] $Y'_j$ is continuous and non-decreasing, 
        \item [(c)] $Y'_j(t)$ can increase only on $F_{f(j)}$, i.e., 
        \begin{align*}
            \int_0^{t}\mathbf{1}{F_{f(j)}}(W'(s))dY'_j(s)=Y'_j(t) \quad \text{for all } t \ge 0
        \end{align*}

    \end{itemize}
\end{enumerate}
\end{lemma}
\begin{proof}
Properties (i), (ii), and (iv)(a)-(b) can be verified directly using Skorokhod's representation theorem \cite[Theorem~1.8, Chapter~3]{ethier2009markov}, and property (iii) follows immediately from assumptions (iii) and (v). Hence, it remains to verify property (iv)(c). We prove this now.

    Suppose that $(W^{n_\ell}, X^{n_\ell}, Y^{n_\ell}) \Rightarrow (W', X', Y')$ as $\ell \to \infty$. Fix $j \in \{1,2,\dots, k\}$ and $t \ge 0$. Our goal is to prove 
    \begin{align*}
        \int^t_0 \mathbf{1}(d(W'(s),F_{f(j)})>0)dY'_j(s) = 0 \quad \text{a.s.}
    \end{align*}
    For each $p =1,2,\dots$, define $g_p:S \to [0,1]$ by 
    \begin{align*}
        g_p(x)= \begin{cases}
            0 \quad &\text{if} \ d(x,F_{f(j)}) \le 1/p, \\
            pd(x,F_{f(j)})-1  \quad &\text{if} \ 1/p < \lVert x \rVert \le 2/p, \\
            1 \quad &\text{otherwise.} 
        \end{cases}
    \end{align*}
    By the monotone convergence theorem, it suffices to show that for each $p =1,2,\dots$, 
    \begin{align}\label{eq: goal}
        \int^t_0 g_p(W'(s))dY'_j(s) = 0 \quad \text{a.s.}
    \end{align}
    
    We also fix $p$. It follows from (iv)(c) that we can take a sufficiently large $L$ such that  
    \begin{align*}
        \int^t_0 g_p(W^{n_\ell}(s))dY^{n_\ell}_j(s) = 0 \quad \text{a.s.}
    \end{align*}
    for $\ell \ge L$. By Skorokhod's representation theorem, we may assume that $(W^{n_\ell}, X^{n_\ell}, Y^{n_\ell})$ almost surely converges to $(W',X',Y')$ uniformly on compact time interval. Since $g_p$ is uniformly continuous,
    \begin{align}\label{eq: 8.1.2}
        g_p(W^{n_\ell}(\cdot)) \to g_p(W'(\cdot)) \quad \text{u.o.c. a.s. as $\ell \to \infty$.}
    \end{align}

    Now, we claim that
    \begin{align*}
        \int^t_0  g_p(W^{n_\ell}(s))dY^{n_\ell}_j(s) \to \int^t_0 g_p(W'(s))dY'_j(s) \quad \text{a.s.},
    \end{align*}
   which, in turn, establishes \eqref{eq: goal} and completes the proof.
    Indeed, write
    \begin{align}\label{eq: 8.1.1}
         \int^t_0  g_p(W^{n_\ell}(s))dY^{n_\ell}_j(s) - \int^t_0 g_p(W'(s))dY'_j(s) \notag
        \\ = \int^t_0  (g_p(W^{n_\ell}(s))-g_p(W'(s)))dY^{n_\ell}_j(s)+\int^t_0 g_p(W'(s))d(Y^{n_\ell}_j-Y'_j)(s)
    \end{align}
    The first term in \eqref{eq: 8.1.1} vanishes a.s. as $\ell \to \infty$ a.s. by \eqref{eq: 8.1.2}.
    For the second term, we approximate $g_p(W'(\cdot))$ by step functions. For each $q = 1,2,\dots$,  
    \begin{align*}
        w^q(s) = g_p(W'(0))\mathbf{1}_{\{0\}}(s)+ \sum_{r=0}^{2^q-1} g_p\left(W'\left(\frac{(r+1)t}{2^q}\right)\right)\mathbf{1}_{(rt/2^q,(r+1)t/2^q]}(s),
    \end{align*}
    for all $0 \le s \le t$. Since $g_p(W'(\cdot))$ is continuous a.s., 
    \begin{align}\label{eq: 8.1.3}
        w^q(\cdot) \to g_p(W'(\cdot)) \quad \text{u.o.c. a.s. as } q \to \infty.
    \end{align}
    
    This yields 
    \begin{align}\label{eq: 8.1.4}
    \begin{split}
        &\hspace{-1cm} \left|\int^t_0 g_p(W'(s))d(Y^{n_\ell}_j-Y'_j)(s)\right| \\
        & \le \left|\int^t_0 (g_p(W'(s))-w^q(s))d(Y^{n_\ell}_j-Y'_j)(s)\right| + \left|\int^t_0 w^q(s)d(Y^{n_\ell}_j-Y'_j)(s)\right| \\
        &\le \sup_{0 \le s \le t}|g_p(W'(s))-w^q(s)|(Y_j^{n_\ell}(t)+Y'_j(t)) \\
          & \qquad +\left|\sum_{r=0}^{2^q-1}g_p\left(W'\left(\frac{(r+1)t}{2^q}\right)\right)\left(Y^{n_\ell}_j\left(\frac{(r+1)t}{2^q}\right)-Y'_j\left(\frac{(r+1)t}{2^q}\right)-\left(Y^{n_\ell}_j\left(\frac{rt}{2^q}\right)-Y'_j\left(\frac{rt}{2^q}\right)\right)\right)\right|.
    \end{split}
    \end{align}
    For fixed $q$, the last term in \eqref{eq: 8.1.4} converges to $0$ a.s. as $\ell\to\infty$ since $Y^{n_\ell}_j(\cdot) \to Y'_j(\cdot)$ u.o.c.
Then letting $q \to \infty$, \eqref{eq: 8.1.3} shows that the first term also vanishes a.s.
\end{proof}

\subsection{Step 3}
Now, our goal is to show that $Y'_j \equiv 0$ for $j>m$. By Lemma~\ref{lem: main1}, this implies that the processes $W',X',\hat Y'$ with reflection matrix $\hat R$ satisfy Definition~\ref{def: srbmgeneral}. Since $j>m$ if and only if $|f(j)| \ge 2$, the claim is equivalent to the following lemma:

\begin{lemma}\label{lem:pushing_process_high_order_vanish}
    If $|f(j)| \ge 2$, $Y'_j \equiv 0$ almost surely. 
\end{lemma}

\begin{proof}
    The argument is a straightforward adaptation of the proofs of Lemmas~4.5 and~4.6 in \cite{dai1996existence}. 
    While those results assume that the reflection matrix $R$ satisfies $\mathcal S$-condition, their proofs in fact only rely on the weaker property that for each maximal subset of $\{1,2,\dots,m\}$, the associated principal sub-matrix of $NR$  has rows spanning a strictly positive vector with positive coefficients, which is generalized as weak $\mathcal S$-condition.  

    Our objective is to establish that 
    \begin{align}\label{eq: codim2}
        \int^\infty_0 \mathbf{1}_{F_I}(W'(s))dY'_j(s) = 0 \quad \text{a.s.}
    \end{align}
    for each subset $I \subseteq \{1,2,\dots, m\}$ with $|I|\ge 2$ and $j \in \{1,2,\dots, k\}$ with $f(j) \subseteq I$. We may assume throughout that $d \ge 2$. 

    We proceed with this proof in two steps. First, we verify \eqref{eq: codim2} for $I = \{1,2,\dots, m\}$.
    Second, we establish the result for all other $I$ by backward induction on $|I|$. Suppose $I \subseteq \{1,2,\dots,m\}$ with $2 \le |I| \le m-1$, and assume that \eqref{eq: codim2} holds for every $I'$ such that $I \subsetneq I' \subseteq \{1,2,\dots,m\}$. 
We then show that \eqref{eq: codim2} also holds for $I$. 
    
Once \eqref{eq: codim2} has been established, it immediately follows that for each $j$ with $|f(j)| \ge 2$, 
\begin{align*}
    \int^\infty_0 \mathbf{1}_{F_{f(j)}}(W'(s))dY'_j(s) = \lim_{t \to \infty}Y'_j(t) = 0,
\end{align*}
where the first equality follows from property (iv)(c) in Lemma~\ref{lem: main1}. Hence $Y'_j \equiv 0$ a.s.

We first show that by fixing $I = \{1,2,\dots, m\}$, for each $j \in \{1,2,\dots, k\}$, 
\begin{align}\label{eq: goal 1}
        \int^\infty_0 \mathbf{1}_{F_I}(W'(s))dY'_j(s) = 0 \quad \text{a.s.}
    \end{align}
If $F_I = \emptyset$, it trivially holds. Hence we assume that $F_I \neq \emptyset$ and fix $x_0 \in F_I$. 

By weak $\mathcal S$-condition of $R$, there exists $\gamma \in \R_{>0}^m$ such that $\gamma^\top NR = \eta^\top R >\mathbf{0}$, where $\eta: = N^\top \gamma$. Define $\delta := (NR)^\top\gamma$ and a matrix $A$ whose rows are those of $N$ and span the row space of $N$ linearly independently. With these parameters, by following \cite[Equations (4.23)--(4.28)]{dai1996existence}, we can construct a family $\{\phi_\eps: \eps \in (0,1)\}$ of $C^2$ functions which are defined on some domain containing $S$, and show that $\{\phi_\eps: \eps \in (0,1)\}$ satisfies the follows: 
\begin{itemize}
    \item[(i)] For each positive integer $n$, $\phi_\eps$ is uniformly bounded on $\{x \in S: \lVert A(x - x_0)\rVert \le n\}$ as $\eps \to 0$. 
    \item[(ii)] $\nabla \phi_\eps$ is bounded on $\{x \in S: \lVert A(x - x_0) \rVert \le n\}$.
    \item[(iii)] $\phi_\eps$ is $L$-harmonic for each $\eps \in (0,1)$.
    \item[(iv)] For each $j = 1,2,\dots, k$, there exist positive constants $\beta_j, c_j, \hat c_j>0$ such that 
    \begin{align*}
        R_j \cdot \nabla \phi_\eps(x) \ge -c_j (\log \eps + 1) \quad \text{for all $x \in S$ with $\lVert A(x - x_0) \rVert <\eps \beta_j$}
    \end{align*}
    and
    \begin{align*}
        R_j \cdot \nabla \phi_\eps(x) \ge -\hat c_j \quad \text{for all $x \in S$}.
    \end{align*}
\end{itemize}

For each positive integer $n$, define
 \begin{align*}
    T_n = \inf\{t \ge 0: \lVert A(W'(t)-x_0) \rVert \ge n \ \text{or} \ Y'_j(t) \ge n \ \text{for some $j \in \{1,2,\dots, k\}$}\}\wedge n,
\end{align*}
Then as in \cite[Equation 4.29]{dai1996existence}, by applying Ito's formula to $\phi_\eps(W(\cdot \wedge T_n))$ and taking expectation on each side, we obtain
\begin{align*}
    \E[\phi_\eps(W'(t \wedge T_n))-\phi_\eps(W'(0))] = \sum^k_{j=1} \E\left[\int^{t \wedge T_n}_0 R_j \cdot \nabla \phi_\eps(W'(s))dY'_j(s)\right] \\
    \ge-(\log \eps+1)\sum^k_{j=1}c_j\E\left[\int^{t \wedge T_n}_0 \mathbf{1}_{\{\lVert A(W'(s) - x_0) \rVert <\eps \beta_j\}}dY'_j(s) \right]-\sum^k_{j=1}\hat c_j \E[Y'_j(t \wedge T_n)],
\end{align*}
where the equality follows from properties (ii) and (iii) and the inequality follows from property (iv). 

By (i) and the definition of $T_n$, the left-hand side is bounded uniformly as $\eps \to 0$.  
Moreover, the second summation on the second line is finite and independent of $\eps$.  
Since $-(\log \eps + 1)\uparrow \infty$ as $\eps\downarrow 0$ and $\E[\int^{t \wedge T_n}_0 \mathbf{1}_{\{\lVert A(W'(s) - x_0) \rVert <\eps \beta_j\}}dY'_j(s)] \ge 0$, we obtain
\begin{align*}
    \E\left[\int^{t \wedge T_n}_0 \mathbf{1}_{F_I}(W'(s))dY'_j(s)\right] &\le\E\left[\int^{t \wedge T_n}_0 \mathbf{1}_{\{\lVert A(W'(s) - x_0) \rVert =0\}}dY'_j(s)\right]\\
    & =\lim_{\eps \to 0}\, \E\left[\int^{t \wedge T_n}_0 \mathbf{1}_{\{\lVert A(W'(s) - x_0) \rVert <\eps \beta_j\}}dY'_j(s)\right] = 0,
\end{align*}
where the first equality follows from Fatou's lemma.  
By letting $m \to \infty$, the monotone convergence theorem yields \eqref{eq: goal 1}.

Lastly, we complete the proof by backward induction on $|I|$.
We may further assume $d \ge 3$ and fix a maximal subset $I \subseteq \{1,2,\dots,m\}$ with $2 \le |I| \le m-1$.
Assume that for every $I'$ such that $I \subsetneq I' \subseteq \{1,2,\dots,m\}$,
\begin{align*}
\int^\infty_0 \mathbf{1}_{F_I}(W'(s))dY'_j(s) = 0 \quad \text{a.s.} 
\end{align*}
for all $j \in \{1,2,\dots,k\}$ with $f(j) \subseteq I'$.

Then, by the induction hypothesis, a.s. for each $j$ with $f(j) \subseteq I$,
\begin{align*}
\int^\infty_0 \mathbf{1}_{F_I}(W'(s))dY'_j(s)
= \int^\infty_0 \mathbf{1}_{F_I}(W'(s)) \mathbf{1}_{\{n_i \cdot W'(s)-b_i \ge 0 \text{ for all $i \not \in I$}\}}dY'_j(s).
\end{align*}
Thus, by the monotone convergence theorem, it suffices to prove that for each $\eps > 0$ and $j$ with $f(j) \subseteq I$,
\begin{align}\label{eq: back1}
\int^\infty_0 \mathbf{1}_{F_I}(W'(s)) \mathbf{1}_{\{n_i \cdot W'(s)-b_i \ge \eps \text{ for all $i \not \in I$}\}}dY'_j(s) = 0
\quad \text{a.s.}
\end{align}

Fix $\eps$ and define a sequence of stopping times $\{\tau_n\}^{\infty}_{n=0}$ by
\begin{align*}
    \tau_{2n-1}: = \inf\{s \ge \tau_{2n-2}: n_i \cdot W'(s) \ge  \eps \text{ for all $i \not \in I$}\}, \quad \tau_{2n}: = \inf\{s \ge \tau_{2n-1}: n_i \cdot W'(s) \le \eps/2 \ \text{for some} \ i \not \in I\}
\end{align*}
for $n \ge 1$ with $\tau_0 := 0$.  Then, for each $j$ with $f(j) \subseteq I$, we have 
\begin{align}\label{eq: back2}
\int^\infty_0 \mathbf{1}_{F_I}(W'(s)) \mathbf{1}_{\{n_i \cdot W'(s)-b_i \ge \eps \text{ for all $i \not \in I$}\}}dY'_j(s)
\le \sum_{n=1}^\infty \int_{\tau_{2n-1}}^{\tau_{2n}} \mathbf{1}_{F_I}(W'(s))dY'_j(s).
\end{align}
Under the  assumption ${\tau_{2n-1}<\infty}$ we can view $W'((\cdot + \tau_{2n-1})\wedge\tau_{2n})$ as a process defined on $S^I$ with reflection matrix $R^I$. By the assumption that $R$ is weak $\mathcal S$ and $I$ is maximal, we can use the method as in the case $I = \{1,2,\dots, m\}$ to obtain 
\begin{align*}
    \int_{\tau_{2n-1}}^{\tau_{2n}} \mathbf{1}_{F_I}(W'(s))dY'_j(s) = 0 \quad \text{a.s.}
\end{align*}

This, combined with \eqref{eq: back2}, implies \eqref{eq: back1} and therefore completes the induction argument.
\end{proof}
\subsection{Step 4}
As argued above, this lemma completes the proof of Theorem \ref{thm: main}.

\begin{lemma}\label{lem: main2}
    SRBMs associated with $(\R^d, \mathbf 0 , \Sigma, \hat R, \delta_0)$ are unique in distribution. 
\end{lemma}

\begin{proof}
    By \cite[Theorem 1.3]{dai1996existence}, it suffices to show that $\hat R$ satisfies $\mathcal S$-condition. We will check condition (b). By the definition of weak $\mathcal S$-condition, for each maximal subset $I \subseteq \{1,2,\dots, m\}$, there exists $\lambda_I \in \R^I_{> 0}$ such that $\lambda_I^\top N^IR^I>\mathbf 0$. Note that $I \subseteq J_I$ since $f(i) = \{i\} \in I$ for all $i \in I$. Hence, it follows that  $\lambda_I^\top (N\hat R)_{I} = (\lambda_I^\top N^I R^I)_I > \mathbf{0}$. This shows that $\hat R$ satisfies $\mathcal S$-condition, as required.
\end{proof}

\section*{Acknowledgement}

The authors thank Ruth Williams and Vadim Gorin for helpful discussions. DK is supported by the NSF RTG grant DMS-1937241. M.K. was supported by the National Research Foundation
of Korea (NRF) grant funded by the Korean government
(MSIT) (Grant Nos. RS-2023-00219980 and 2022R1C1C1008491, PI: Jinsu Kim, POSTECH).
EL is supported by the NSF RTG grant DMS-2134107. 
HL is partially supported by NSF grant DMS-2206296.
\vspace{0.3cm}

	\small{
		\bibliographystyle{amsalpha}
		\bibliography{mybib}
	}
\end{document}